\documentclass[a4paper,11pt]{article}

\usepackage{subfigure}
\usepackage{amsfonts}
\usepackage{amssymb}
\usepackage{amsthm}
\usepackage{amsmath}
\usepackage{a4wide}
\usepackage{mathrsfs}
\usepackage{epsfig}
\usepackage{palatino}
\usepackage{tikz,fp,ifthen}
\usepackage{float}
\usepackage[all]{xy}
\usepackage{nicefrac}

\usepackage{graphicx}

\usepackage{color}
\definecolor{hanblue}{rgb}{0.27, 0.42, 0.81}
\definecolor{red}{rgb}{1.0, 0.0, 0.0}
\usepackage[colorlinks, citecolor=hanblue,linkcolor=red]{hyperref}

\usepackage{mathtools}
\usepackage{hyperref}
\newcommand{\Tau}{\mathcal{T}}

\renewcommand{\P}{\mathcal{P}}
\newcommand{\J}{\mathcal{J}}
\newcommand{\R}{\mathbb{R}}
\newcommand{\N}{\mathcal{N}}
\newcommand{\E}{\mathcal{E}}
\newcommand{\F}{\mathcal{F}}
\newcommand{\eps}{\varepsilon}

\newcommand{\Sing}{\mathrm{Sing}}
\newcommand{\Reg}{\mathrm{Reg}}

\numberwithin{equation}{section}

\theoremstyle{plain}

\newtheorem{thm}{Theorem}[section]
\newtheorem{lem}[thm]{Lemma}
\newtheorem{prop}[thm]{Proposition}
\newtheorem{cor}[thm]{Corollary}
\newtheorem{ackn}{Acknowledgments\!}

\theoremstyle{definition}
\newtheorem{defn}[thm]{Definition}

\newtheorem{prob}[thm]{Problem}

\theoremstyle{remark}
\newtheorem{rem}[thm]{Remark}

\newtheorem{Example}[thm]{Example}

\numberwithin{equation}{section}

\definecolor{deepjunglegreen}{rgb}{0.0, 0.29, 0.29}
\definecolor{hanblue}{rgb}{0.27, 0.42, 0.81}
\definecolor{yellow(ncs)}{rgb}{1.0, 0.83, 0.0}
\definecolor{brightpink}{rgb}{1.0, 0.0, 0.5}

\usetikzlibrary{decorations.pathmorphing,backgrounds,fit,calc,through}
\def\centerarc[#1](#2)(#3:#4:#5)% [draw options] (center) (initial angle:final angle:radius) 
{
 \draw[#1] ($(#2)+({#5*cos(#3)},{#5*sin(#3)})$) arc (#3:#4:#5); }
 
\title{Degenerate elastic networks}

\author{Giacomo Del~Nin
\footnote{Giacomo.Del-Nin@warwick.ac.uk, Mathematics Institute, University of Warwick, 
Zeeman Building, CV4 7HP Coventry, UK.}
 \and Alessandra Pluda
\footnote{alessandra.pluda@unipi.it, Dipartimento di Matematica, Universit\`a di Pisa, Largo Bruno Pontecorvo 5, 56127 Pisa, Italy.}
\and Marco Pozzetta
\footnote{pozzetta@mail.dm.unipi.it, Dipartimento di Matematica, Universit\`a di Pisa, Largo Bruno Pontecorvo 5, 56127 Pisa, Italy.}
}

\begin{document}

\maketitle

\begin{abstract}
\noindent We minimize a linear combination of the length and the $L^2$-norm of the curvature among networks in $\mathbb{R}^d$ belonging to a given class determined by the number of curves, the order of the junctions, and the angles between curves at the junctions. Since this class lacks compactness, we 
characterize the set of limits of sequences of networks bounded in energy, providing an explicit representation of the relaxed problem. This is expressed in terms
of the new notion of degenerate elastic networks that, rather surprisingly, involves only the properties of the given class, without reference to the curvature. In the case of $d=2$ we also give an equivalent description of degenerate elastic networks by means of a combinatorial definition easy to validate by a finite algorithm. Moreover we provide examples, counterexamples, and additional results that motivate our study and show the sharpness of our characterization.
\end{abstract}



%\tableofcontents
%

\section{Introduction}

A regular network $\mathcal{N}$ is a connected set in $\mathbb{R}^d$ composed of $N$
regular curves $\gamma^i$ of class $H^2$  that meet at their endpoints
in junctions of possibly different order. 
Moreover the angles at the junctions are assigned by a fixed set of directions $\mathcal{D}$ 
as we will define more precisely in Definition \ref{angolifissati}.
The \emph{elastic energy functional} $\mathcal{E}$
for a network $\mathcal{N}$ is given by
%defined as
\begin{equation}\label{Def:energiaintro}
\mathcal{E}\left(\mathcal{N}\right):=
\sum_{i=1}^N\left(\int_{\gamma^{i}} \vert \vec{k}^i\vert^{2} \,\mathrm{d}s
+\ell(\gamma^i)\right)\,,
\end{equation}
where $\vec{k}^i$ is the curvature, $s$ the arclength parameter 
and $\ell(\gamma^i)$ is the length
of the curve $\gamma^i$.\\

The elastic energy functional has a long history.
Already at the times of Galileo scientists
tried to model elastic rods and strings, looking
for 
equations for equilibrium of moments and forces.
The idea to relate the curvature 
of the fiber of the beam to the bending moment  came only later
when, in 1691, Jacob Bernoulli  proposed to model the bending energy 
of thin inextensible elastic rods
with a functional involving the curvature.
Several authors refer to the functional~\eqref{Def:energiaintro} as
Euler Elastic energy in honour of Euler (while we will simply call it elastic energy)
who solved the problem of minimizing 
the potential energy of the elastic laminae
using  variational techniques. 
Even nowadays the elastic energy
appears in several mechanical and physical models 
(c.f.\cite{Truesdell}) 
and in imaging sciences, see for instance~\cite{Mumford1994}.

We are interested in the minimization of the  functional $\mathcal{E}$
among networks with fixed topology and with fixed 
angles at the junctions assigned by $\mathcal{D}$.

Notice that the length 
$\ell(\gamma^i)$ 
of each curve of a regular network
is strictly positive because the curves are regular by assumption.
This property is not preserved by sequences with uniformly bounded energy:
along a sequence of networks $\{\N_n\}_{n\in\mathbb{N}}$ with uniformly bounded energy
 (or even a minimizing sequence) 
 the length of a curve may go to zero as $n\to\infty$,
producing a  ``degenerate limit" which is no longer a regular network.
Hence a remarkable issue is to understand this lack of compactness 
of sequences with bounded energy in minimization problems
\footnote{By compactness we mean sequential compactness, where 
the natural notion of convergence for the sequences of networks is the 
weak $H^2$ convergence.}.

Our first task is therefore to describe the class of 
limits of sequences (equibounded in energy) of regular networks.
In other words, we have to 
characterize the ``closure in energy"
of the class $\mathcal{C}_{\mathrm{Reg}}$ of regular networks.
%(briefly $\overline{\mathcal{C}_{\mathrm{Reg}}}$).

Then, since the lack of compactness implies that it could not be 
possible to solve the original minimization problem among regular networks,
we relax it by considering
the lower semicontinuous envelope of the functional $\mathcal{E}$ with respect to the weak convergence in $H^2$.
% namely 
% \begin{equation}\label{lowersemicontenv}
% \inf
% \left\lbrace 
% \liminf_{n\to\infty}\mathcal{E}(\mathcal{N}_n)\;\vert\; 
% \mathcal{N}_n\in\mathcal{C}_{\mathrm{Reg}}
% %\,, \mathcal{N}_n\overset{H^2}{\rightharpoonup}\mathcal{N}_\infty
% \right\rbrace 
% \,.
% \end{equation}
Our second objective is to find 
an explicit formula of the  lower semicontinuous envelope.
We underline that 
the relaxation of the problem is necessary:
Example~\ref{esempiodeg} shows that 
also in  some very simple situations
minimizers are not regular networks.

\medskip

In Proposition~\ref{regolariconvergonoadeg}
and Proposition~\ref{recovery}
we characterize
the smallest compact class of (non--regular) networks in which 
the class of regular networks is
 ``dense in energy" 
in terms of a mixture of algebraic and 
combinatorial conditions that are easy to verify and rely on the 
topological assumptions on the
competitors. As the formulation of these conditions involves
some technicalities,  a detailed description is postponed to the second part of the introduction. 
We are then able to give a characterization of the lower semicontinuous envelope of  
$\mathcal{E}$ that strongly relies on the previous result (see Theorem~\ref{rilassato}).

\medskip

This last point creates a bridge between our problem and 
another natural research direction:
the extension of the elastic energy
(and more in general of functionals that depend on the curvature) to singular sets 
for which at least a \emph{weak} notion of curvature is well defined, varifolds for instance. 
The characterization of the relaxation of the original functional in such cases
turned out  to be particularly difficult (for results in this direction see 
for instance~\cite{bellettinipaolini,bellettinidalmasopaolini,bellettinimugnai04,bellettinimugnai07,
masnounardi13a,masnounardi13b,pozzetta2}). 
We notice that networks can be seen as a simple example of sets which are essentially singular
and they can be understood as a \emph{boundary} of a planar cluster of surfaces.
For problems in which one knows a priori that the boundary 
of minimal planar clusters is composed of a 
finite number of curves, it could be useful to define the curvature
of the cluster by means of~\eqref{Def:energiaintro}.
An advantage is that we have an easy integral representation 
of the lower semicontinuous envelope of $\mathcal{E}$.

To place our paper in the broader context, 
we mention the minimization of Willmore-type functionals among both Riemannian manifolds 
and singular structures is an extreme flourishing research field.
As the simplest possible example of singular structure we mention 
surfaces with boundary (possible references 
are~\cite{alessandronikuwert,deckerlnickgrunauroeger,dondl,
menzel,pozzetta,schaetzle}).

The main reason for which we got interested in this problem is a previous
study of its dynamical counterpart (see~\cite{garckemenzelpluda, garckemenzelpluda2}).
The study of the static problem has often revealed to be useful for the analysis 
of the asymptotic  behavior of the solutions of the associated gradient flow and of the singularities
that can appear during the evolution. 
%On the contrary, one can study the flow
%to find critical points of the functional. 
Our analysis can be useful to understand the long
time behavior of the elastic flow of networks
whose curves meet at junctions 
with prescribed angles introduces in~\cite{baganu,garckemenzelpluda}.

\medskip

After this detour on the literature,
we explain more in detail 
what can happen along sequences of regular networks
$\{\N_n\}_{n\in\mathbb{N}}$ whose elastic energy is uniformly bounded.

\smallskip

\textbf{A model problem: Theta--networks}

\smallskip

A Theta--network is a regular 
$3$--network whose curves form equal angles at the two junctions 
(Figure~\ref{fig:theta}). 
The limit of a sequence of regular Theta--networks with uniformly bounded
energy may not exist in the class of regular networks and hence the minimization problem has 
some form of degeneracy that luckily is not completely wild:
%{\color{blue} as we shall see,}\footnote{\green{Pensavo si riferisse al lavoro suo, quindi toglierei anche il as we shall see}} 
the length of at most one curve can go to zero along the sequence,
becoming straighter and straighter. 
If this is the case the ``degenerate" limit network is composed
of two curves meeting at a quadruple point 
forming angles equal in pairs of $\tfrac{\pi}{3}$ and $\tfrac{2\pi}{3}$.

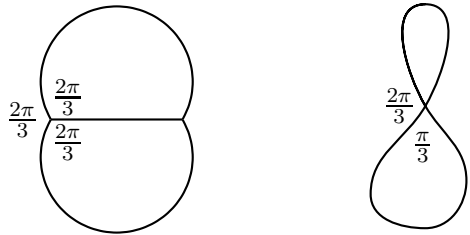
\begin{figure}[H]
\begin{center}
\begin{tikzpicture}[scale=1]
\draw[thick,color=black,scale=1,domain=-2.09: 2.09,
smooth,variable=\t,shift={(0,0)},rotate=0]plot({1.*sin(\t r)},
{1.*cos(\t r)}) ; 
\draw[thick, color=black,scale=1,domain=-2.09: 2.09,
smooth,variable=\t,shift={(0,-1)},rotate=180]plot({1.*sin(\t r)},
{1.*cos(\t r)}) ; 
\draw[thick]
(-0.87,-0.5)--(0.87,-0.5);
\draw[white]
(0,-2.3)--(0.1,-2.3);
\path[font=\normalsize]
(-0.3,-0.8)node[left]{$\tfrac{2\pi}{3}$}
(-0.3,-0.2)node[left]{$\tfrac{2\pi}{3}$}
(-0.9,-0.5)node[left]{$\tfrac{2\pi}{3}$};
\end{tikzpicture}\qquad\qquad\qquad
\begin{tikzpicture}[scale=1.8]
\draw[white]
(0,-1.1)--(1,-1.1);
\draw[thick]
(0,0)to[out= 60,in=0, looseness=1] (0,0.75)
(0,0.75)to[out= 180,in=120, looseness=1] (0,0)
(0,0)to[out= -60,in=90, looseness=1] (0.3,-0.55)
(0.3,-0.55)to[out= -90,in=0, looseness=1] (0,-0.9)
(0,0.75)to[out= 180,in=120, looseness=1] (0,0)
(0,0)to[out= -120,in=90, looseness=1] (-0.4,-0.65)
(-0.4,-0.65)to[out= -90,in=180, looseness=1] (0,-0.9);
\path[font=\normalsize]
(0,0)node[left]{$\tfrac{2\pi}{3}$}
(0.13,-0.3)node[left]{$\tfrac{\pi}{3}$};
\end{tikzpicture}
\end{center}
\caption{A Theta--network and a degenerate Theta--network}\label{fig:theta}
\end{figure}

\smallskip

The general case 
%of regular $N$--networks
presents new interesting features with respect to the 
model problem of Theta--networks.
Firstly, one realizes that  much more than the length of one single curve can go to zero, 
since entire parts of the network can vanish as $n\to\infty$.
In order to keep track of the complexity of the network we introduce the notion
of underlying graph $G$.

\smallskip

\textbf{The underlying graph}

\smallskip

The oriented graph  $G$ is
composed of edges $E_i$ whose endpoints are identified in vertices
of possibly different order (Definition~\ref{defgraph}).
A network then is a pair graph--continuous map $\N=(G,\Gamma)$
with $\Gamma:G\to\mathbb{R}^d$ (Definition~\ref{network}).
The underlying graph $G$ captures the topology of the networks
and allows us to have a reminiscence of 
their structure even when some curves in a sequence of networks
$\N_n=(G,\Gamma_n)$ collapse to a point in the limit.

\smallskip

\textbf{A necessary angle condition}

\smallskip

We say that a curve $\gamma^i:=\Gamma\vert_{E_i}$ of a network is \emph{singular}
if it is a constant map,
and  regular if it is an immersion of class $H^2$. 
A network is singular if some of its curves are singular.

\begin{defn}[Outer tangents]\label{outertang}
Consider a regular curve of a network $\N$ parametrized by $\gamma^i$.
We define $\tau^{z,i}$, with $z\in\{0,1\}$, as
the outer tangent vector 
at the endpoint $\gamma^i(z)$ to the curve $\gamma^i$ given by
\begin{equation*}
\tau^{z,i}=(-1)^{z}\frac{\dot\gamma^i(z)}{\vert\dot\gamma^i(z)\vert}\,.
\end{equation*}
\end{defn}

We remark that the vector $\tau^{z,i}$ ``points inside'' the curve $\gamma^i$ at $\gamma^i(z)$ (see Figure \ref{fig:OuterTangents}). In this way, if $\gamma^i(z_i)=\gamma^j(z_j)$ is a junction point, then $\tau^{z_i,i}$ and $\tau^{z_j,j}$ ``point outwards'' \emph{with respect to the junction}. In this way the outer tangents are geometrically independent of the direction of the parametrization of the curves concurring at a junction.

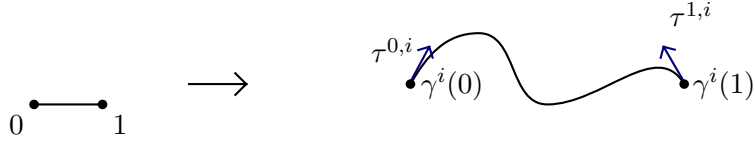
\begin{figure}[H]
\begin{center}
\begin{tikzpicture}[scale=0.9]
\draw[color=black, thick]
(-4.5,-0.3)--(-3.5,-0.3);
\draw[ thick]
(1,0)to[out= 60,in=180] (2,0.75)
(2,0.75)to[out= 0,in=180] (3,-0.3)
(3,-0.3)to[out= 0,in=120] (5,0);

\draw[thick, shift={(-1.4,0)}, rotate=-90, scale=1.7, black]
(0,0)to[out= -45,in=135, looseness=1] (0.1,-0.1)
(0,0)to[out= -135,in=45, looseness=1] (-0.1,-0.1)
(0,-0.5)to[out= 90,in=-90, looseness=1] (0,0.013);

\draw[thick, shift={(1.275,0.55)}, rotate=-28, scale=1.3, blue!50!black]
(0,0)to[out= -45,in=135, looseness=1] (0.1,-0.1)
(0,0)to[out= -135,in=45, looseness=1] (-0.1,-0.1)
(0,-0.5)to[out= 90,in=-90, looseness=1] (0,0.013);

\draw[thick, shift={(4.7,0.55)}, rotate=30, scale=1.3, blue!50!black]
(0,0)to[out= -45,in=135, looseness=1] (0.1,-0.1)
(0,0)to[out= -135,in=45, looseness=1] (-0.1,-0.1)
(0,-0.5)to[out= 90,in=-90, looseness=1] (0,0.013);

\fill[black]
(-4.5,-0.3)circle(2pt);
\fill[black]
(-3.5,-0.3)circle(2pt);

\fill[black]
(5,0) circle(2pt);
\fill[black]
(1,0) circle(2pt);

\path[font=\normalsize]

(-4.5,-0.6)node[left]{$0$}
(-3.5,-0.6)node[right]{$1$}

(6.2,0)node[left]{$\gamma^i(1)$}
(1,0)node[right]{$\gamma^i(0)$}
(1.2,0.5)node[left]{$\tau^{0,i}$}
(5.1,0.7)node[above]{$\tau^{1,i}$};
\end{tikzpicture}
\caption{Outer tangents at the endpoints of the curve $\gamma^i$.}\label{fig:OuterTangents}
\end{center}
\end{figure}

The following definition introduces the \emph{angle condition} for a possibly singular network. It will be restated more in detail in Definition \ref{angolifissatiRd}.

\begin{defn}[Angle condition]\label{angcondshort}
A (possibly) singular network $\mathcal{N}=(G, \Gamma)$
satisfies the angle condition if for every singular curve $\gamma^i$ there exist 
two unit vectors $\tau^{z,i}$, called its virtual tangents, with $z\in\{0,1\}$
such that $\tau^{0,i}=-\tau^{1,i}$ and
such that, at each junction, the tangent vectors (either real outer tangents or virtual ones)
coincide, up to a fixed rotation that only depends on the junction, with the directions 
assigned by a given set $\mathcal{D}$ (we shall state this concept more in detail in Definition~\ref{angolifissati} and Definition~\ref{angolifissatiRd}).
\end{defn}

%A sequence of regular network with uniformly bounded energy satisfies 
%the angle condition~\ref{angcondshort} for $n\to \infty$.

Also motivated by~\cite{danovplu},
as a first attempt we defined the class of degenerate networks
as all the singular networks that satisfy the angle condition of Definition~\ref{angcondshort}.

Unfortunately this \emph{purely algebraic} condition on the angles 
turns out to be necessary but not sufficient to characterize the closure of the class 
$\mathcal{C}_{\mathrm{Reg}}$ as we show in the next example.

Consider the topology depicted in Figure~\ref{cicloproblematico}
on the left, with all junctions forming angles of $\tfrac{2\pi}{3}$, and define a singular network that collapses the red part to a single point. It is possible to do so satisfying the angle 
condition of Definition~\ref{angcondshort} (see the picture on the right) but no sequence of regular networks with finite energy and the same angle constraint could converge to the singular network on the right.

\begin{figure}[H]
\begin{center}
% \begin{tikzpicture}[scale=0.7]
% \draw[color=red, thick]
% (1,-0.86)to[out= 0,in=0](-1.5,2)
% (-1.5,2)to[out= 180,in=180](-1,0.86)
% (-1,0.86)--(-0.5,0)
% (-0.5,0)--(0.5,0)
% (0.5,0)--(1,-0.86);
% \draw[thick]
% (-0.5,1.72)to[out= 60,in=0] (-1.3,2.75)
% (-1.3,2.75)to[out= 180,in=-120] (-1,-0.86)
% (-1,0.86)--(-0.5,1.72)
% (-0.5,0)--(-1,-0.86);
% \draw[rotate=180,  thick]
% (-0.5,1.72)to[out= 60,in=0] (-1.3,2.75)
% (-1.3,2.75)to[out= 180,in=-120] (-1,-0.86)
% (-1,0.86)--(-0.5,1.72)
% (-0.5,0)--(-1,-0.86);
% \end{tikzpicture}
\begin{tikzpicture}[scale=0.8]
\draw[color=red, thick]
(1,-0.86)to[out= 0,in=0](-1.5,2)
(-1.5,2)to[out= 180,in=180](-1,0.86)
(-1,0.86)--(-0.5,0)
(-0.5,0)--(0.5,0)
(0.5,0)--(1,-0.86);
\draw[ thick]
(-0.5,1.72)to[out= 60,in=0] (-1.3,2.75)
(-1.3,2.75)to[out= 180,in=-120] (-1,-0.86)
(-1,0.86)--(-0.5,1.72)
(-0.5,0)--(-1,-0.86);
\draw[rotate=180,  thick]
(-0.5,1.72)to[out= 60,in=0] (-1.3,2.75)
(-1.3,2.75)to[out= 180,in=-120] (-1,-0.86)
(-1,0.86)--(-0.5,1.72)
(-0.5,0)--(-1,-0.86);
\path[font=\normalsize]
(-0.7,0.5)node[left]{$E_2$}
(-0.5,0.3)node[right]{$E_3$}
(1.12,-0.8)node[above]{$E_4$}
(1,1.5)node[left]{$E_1$};
\end{tikzpicture}
\quad
\begin{tikzpicture}[scale=0.7]
%\draw[color=red, thick]
%(1,-0.86)to[out= 0,in=0](-1.5,2)
%(-1.5,2)to[out= 180,in=180](-1,0.86)
%(-1,0.86)--(-0.5,0)
%(-0.5,0)--(0.5,0)
%(0.5,0)--(1,-0.86);
\draw[white]
(0,-2.7)--(0.1,-2.7);

\draw[thick]
(0,0)to[out= 60,in=0] (-0.8,1.03)
(-0.8,1.03)to[out= 180,in=180] (-1.5,-1.72)
(-1.5,-1.72)to[out= 0,in=-120] (-0.5,-0.86)
(-0.5,-0.86)--(0,0)
;

\draw[color=red, fill] (0,0) circle [radius=0.1];
\draw[rotate=180,  thick]
(0,0)to[out= 60,in=0] (-0.8,1.03)
(-0.8,1.03)to[out= 180,in=180] (-1.5,-1.72)
(-1.5,-1.72)to[out= 0,in=-120] (-0.5,-0.86)
(-0.5,-0.86)--(0,0)
;
\end{tikzpicture}
\end{center}
\caption{Starting from the represented graph $G$, it is possible to construct an example
from which one deduces that the angle condition is not sufficient to define the closure of the  class 
$\mathcal{C}_{\mathrm{Reg}}$. See Example \ref{condnecessarianonsuff} for details.}\label{cicloproblematico}
\end{figure}
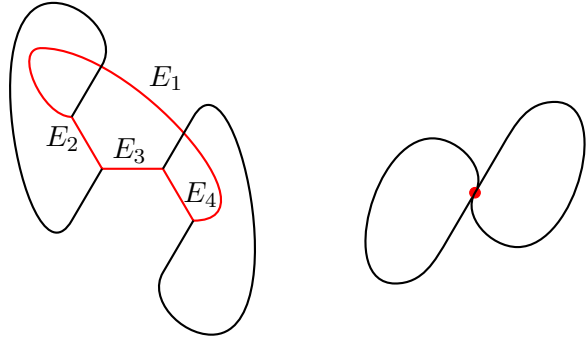

Indeed, denoting by $C$ the cycle in red in Figure \ref{cicloproblematico}, 
using Gauss--Bonnet Theorem (Remark \ref{rem:GaussBonnet})
one gets $\mathcal{E}(C)\geq \frac{c}{L(C)}$,
with $c>0$, for any immersion $\Gamma$ such that $(G,\Gamma)$ is a regular network, and 
thus the energy diverges if the length of the red loop goes to zero
(see Example~\ref{condnecessarianonsuff} and Remark \ref{rem:GaussBonnet} and for more details).

\smallskip

\textbf{Stratified straight subgraph}

\smallskip

To overcome this issue in the definition of the class,
we looked for an extra condition.  Again 
the underlying graph $G$ helps us.

\begin{defn}[Stratified straight subgraph]\label{shortstratstraightRd}
A subgraph $H$ of the underlying graph $G$ is stratified--straight 
if there exists a finite sequence of subgraphs,
called \emph{strata},
\begin{equation*}
\emptyset=H_q\subset H_{q-1}\subset\ldots\subset H_1\subset H_0=H
\end{equation*}
and maps $\Sigma_j:H_j\to \R^d$  
%such that $(H_0,\Sigma_0)$ is a 
%(possibly singular) network satisfying the angle condition 
%in the sense of Definition~\ref{angolifissatiRd} 
%%with (real or fictitious) tangents $\tau^{z,i}$ 
%and 
such that for $j=0,\ldots,q-1$ 
\begin{itemize}
\item the (sub)network  $(H_j,\Sigma_j)$ satisfies the angle condition~\ref{angcondshort}
with (real or virtual) tangent vectors coinciding with the ones associated to $(H_0,\Sigma_0)$ 
and its regular curves are straight segments;
\item the stratum $H_{j+1}$ is the union of the singular curves of $(H_j,\Sigma_j)$.
\end{itemize}
\end{defn}

In the following, segments are always understood to be \emph{straight} segments.\\

\noindent Denote by $H$ the subgraph given by
the union of the edges $E_i$ of $G$ such that the curves $\gamma^i=\Gamma\vert_{E_i}$
are singular.
With a little abuse of notation we say that the network
$\N=(G,\Gamma)$ is stratified straight if $H$ is stratified straight.

\smallskip

\textbf{The class of degenerate networks}

\smallskip

Finally we come to the characterization: a network is degenerate, of class $\mathcal{C}_{\mathrm{Deg}}$,
if it satisfies the angle condition~\ref{angcondshort} and it is stratified straight.
It is natural now to define the extension $\overline{\mathcal{E}}$  of 
the functional $\mathcal{E}$ by setting that a  
curve with zero length of a network in $\mathcal{C}_{\mathrm{Deg}}$ 
gives a null contribution to the energy and  by assigning value $+\infty$ to the energy of 
singular networks which are not in $\mathcal{C}_{\mathrm{Deg}}$.

In Proposition~\ref{regolariconvergonoadeg} we prove that a sequence
of regular networks with equibounded energy converges 
%weakly in $H^2$
to a  network in $\mathcal{C}_{\mathrm{Deg}}$,
namely that the closure
% in energy 
of $\mathcal{C}_{\mathrm{Reg}}$
\emph{is contained} into $\mathcal{C}_{\mathrm{Deg}}$.

\smallskip

\textbf{Formula for the lower semicontinuous envelope of the elastic energy functional}

\smallskip

A question still remains:
is the class $\mathcal{C}_{\mathrm{Deg}}$ the smallest set of generalized network where we 
have compactness?
In Proposition~\ref{recovery} we show that
any degenerate network can be approximated in energy by a regular one. 
Hence our guess was correct:
the class $\mathcal{C}_{\mathrm{Reg}}$ is 
``dense in energy" in $\mathcal{C}_{\mathrm{Deg}}$.
Moreover the extension $\overline{\mathcal{E}}$
is the lower semicontinuous envelope of $\mathcal{E}$.

We stress the fact that we are able to
express the  lower semicontinuous envelope in integral form
and with the very simple formula~\eqref{estensione}.
This is possible only thanks to the precise characterization
of $\mathcal{C}_{\mathrm{Deg}}$.

Notice that the existence of minimizers 
of the relaxed functional $\overline{\mathcal{E}}$ follows trivially by a
direct method in the Calculus of Variations.

\smallskip

\textbf{Stratified straight versus straight subgraph}

\smallskip

At first sight considering stratified straight subgraphs could seem uselessly complicated.\\
One can ask whether the class of degenerate networks can be equivalently characterized as the set of networks $(G,\Gamma)$ satisfying the angle condition~\ref{angcondshort} such that there exists only \emph{one} map $\Sigma:H\to\mathbb{R}^d$ from the singular edges to $\mathbb{R}^d$ such that every curve of $\Sigma$ is a regular straight segment and $(H,\Sigma)$ satisfies the angle condition (in this case we say that the network $(H,\Sigma)$ is straight). 
This would mean that in Definition~\ref{shortstratstraightRd} the index $q=1$ is enough. 
In Example~\ref{limitstrstr} we show a stratified straight but not straight graphs. Nonetheless, there are some cases in which
every stratified straight subgraph is in fact straight; an easy example consists in graphs $G$ with the structure of a tree.\\
However, the concepts of stratified straight graphs and straight graphs are essentially different, 
even in non--trivial cases that have some interest from a variational viewpoint; more precisely in 
Section~\ref{strstrcoinceideconstr} we study the case of networks in $\R^2$ with graphs having junctions of order at most four
such that at these junctions the tangent vectors are orthogonal (Proposition~\ref{prop:QuadratiStraight}) and we find a characterization of the stratified straight graphs that turn out to be straight.

\smallskip

\textbf{Ambient space: restriction to the plane and extension to Riemannian surfaces}

\smallskip

Although we find our characterization
of the class $\mathcal{C}_{\mathrm{Deg}}$ by combination of algebraic and geometric conditions
quite satisfactory, if the ambient space is $\mathbb{R}^2$
we are able to give an equivalent characterization of the class $\mathcal{C}_{\mathrm{Deg}}$
(see Definition~\ref{degnetwRd})
by an algebraic--topological condition that
has the great advantage of being verifiable by an algorithm with finitely many steps.

It is also worth to mention that our result extends from the ambient space $\mathbb{R}^2$
to any $2$--dimensional closed surface embedded in $\mathbb{R}^3$.

\smallskip

\textbf{Fixed lengths}

\smallskip

In Section~\ref{fixedlength} we present 
a variant of the problem, that is analytically
very simple, but  that could be more relevant from the physical point of view:
the minimization of the $L^2$-norm of the curvature among
networks whose curves have fixed length.

\medskip

\textbf{Plan of the paper} 

\smallskip

The structure of the paper is the following:
after the definitions of networks and elastic energy contained in Section~\ref{definitionoftheproblem},
in Section~\ref{cptn} we introduce the class of degenerate networks (Definition~\ref{degnetwRd}).
The compactness of the class of degenerate networks is proved in
Proposition~\ref{regolariconvergonoadeg}.
In Section~\ref{relax}
 we finish the proof of the representation of the relaxed functional by constructing recovery sequences of degenerate networks (Proposition~\ref{recovery}).
Section~\ref{charact} is devoted to the equivalent algorithmic characterization of the class of degenerate networks
when Problem~\ref{problem} is set in $\mathbb{R}^2$.
Then in Section~\ref{strstrcoinceideconstr} we study a non-trivial case in which
there exist stratified straight subgraphs that are not straight and we characterize such difference (Proposition~\ref{prop:QuadratiStraight}).
Subsequently we comment on the extension of our result from the ambient space $\mathbb{R}^2$
to any $2$--dimensional closed surface embedded in $\mathbb{R}^3$. In Section \ref{sec:SideRemarks} we collect some final observations and remarks.
We conclude the paper with the Appendix~\ref{criticalpoints}
where we compute the Euler--Lagrange equations satisfied by critical points of the energy and we prove that critical points are real analytic, up to reparametrization.

\begin{ackn}
The  research of the second and of the third  author has  been  partially  supported  by
INdAM -- GNAMPA Project 2019 ``Problemi geometrici per strutture singolari"
(Geometric problems for singular structures).
This project has received funding from the European Research Council (ERC) under the European Union’s Horizon 2020 research and innovation programme under grant agreement No 757254 (SINGULARITY).
\end{ackn}

%%%%%%%%%%%%%%%%%%%%%%%%%%%%%%%%%%%%%%%%%%%%%%%%%%
%%%%%%%%%%%%%%%%%%%%%%%%%%%%%%%%%%%%%%%%%%%%%%%%%%
\section{Setting and definition of the problem}\label{definitionoftheproblem}
%%%%%%%%%%%%%%%%%%%%%%%%%%%%%%%%%%%%%%%%%%%%%%%%%%
%%%%%%%%%%%%%%%%%%%%%%%%%%%%%%%%%%%%%%%%%%%%%%%%%%

We begin by defining the mathematical objects of our interest.

%%%%%%%%%%%%%%%%%%%%%%%%%%%%%%%%%%%%%%%%%%%%%%%%%%
\subsection{Elastic energy functional for networks}
%%%%%%%%%%%%%%%%%%%%%%%%%%%%%%%%%%%%%%%%%%%%%%%%%%

Fix $N\in\mathbb{N}$, $d\in\mathbb{N}$ with $d\ge 2$ and let $i\in\{1,\ldots, N\}$, $I_i:=[0,1]\times\{i\}$,
$E:=\bigcup_{i=1}^N I_i$ and 
$V:=\bigcup_{i=1}^N \{0,1\}\times\{i\}$.

\begin{defn}($N$--graph)\label{defgraph}
Let $\sim$ be an equivalence relation that identifies points of $V$.
An $N$--graph $G$ is the topological quotient space of $E$ induced by $\sim$, that is
\begin{equation*}
G:=E\Big{/}\sim\,.
\end{equation*}
A subgraph $H\subseteq G$ is the quotient
\begin{equation*}
H:=\bigcup_{j=1}^M I_{i_j}\Big{/}\sim
\end{equation*}
for a given choice of indices $i_1,\ldots,i_M\in\{1,\ldots,N\}$,
where $\sim$ is the equivalence relation defining $G$.
\end{defn}

\begin{figure}[H]
\begin{center}
\begin{tikzpicture}[scale=0.9]

\draw[color=black, thick]
(-4.5,-0.3)--(-3.5,-0.3)
(-4.5,-1.3)--(-3.5,-1.3)
(-4.5,-2.3)--(-3.5,-2.3)
(-4.5,-3.3)--(-3.5,-3.3)
(-2.5,-0.3)--(-1.5,-0.3)
(-2.5,-1.3)--(-1.5,-1.3)
(-2.5,-2.3)--(-1.5,-2.3)
(-2.5,-3.3)--(-1.5,-3.3)

(3,-0.8)--(4,-0.8)
(3,-0.8)--(3,-1.8)
(4,-0.8)--(4,-1.8)
(2,-1.8)--(4,-1.8)
(3,-2.8)--(3,-1.8)
(2,-2.8)--(2,-1.8)
(2,-2.8)--(3,-2.8);

\draw[thick, shift={(2,-2.3)}, rotate=-180, scale=1, black]
(0,0)to[out= -45,in=135, looseness=1] (0.1,-0.1)
(0,0)to[out= -135,in=45, looseness=1] (-0.1,-0.1);

\draw[thick, shift={(3,-2.3)}, rotate=-180, scale=1, black]
(0,0)to[out= -45,in=135, looseness=1] (0.1,-0.1)
(0,0)to[out= -135,in=45, looseness=1] (-0.1,-0.1);

\draw[thick, shift={(3,-1.3)}, rotate=-180, scale=1, black]
(0,0)to[out= -45,in=135, looseness=1] (0.1,-0.1)
(0,0)to[out= -135,in=45, looseness=1] (-0.1,-0.1);

\draw[thick, shift={(4,-1.3)}, rotate=-180, scale=1, black]
(0,0)to[out= -45,in=135, looseness=1] (0.1,-0.1)
(0,0)to[out= -135,in=45, looseness=1] (-0.1,-0.1);

\draw[thick, shift={(3.5,-0.8)}, rotate=-90, scale=1, black]
(0,0)to[out= -45,in=135, looseness=1] (0.1,-0.1)
(0,0)to[out= -135,in=45, looseness=1] (-0.1,-0.1);

\draw[thick, shift={(3.5,-1.8)}, rotate=-90, scale=1, black]
(0,0)to[out= -45,in=135, looseness=1] (0.1,-0.1)
(0,0)to[out= -135,in=45, looseness=1] (-0.1,-0.1);

\draw[thick, shift={(2.5,-1.8)}, rotate=-90, scale=1, black]
(0,0)to[out= -45,in=135, looseness=1] (0.1,-0.1)
(0,0)to[out= -135,in=45, looseness=1] (-0.1,-0.1);

\draw[thick, shift={(2.5,-2.8)}, rotate=-90, scale=1, black]
(0,0)to[out= -45,in=135, looseness=1] (0.1,-0.1)
(0,0)to[out= -135,in=45, looseness=1] (-0.1,-0.1);

\fill[black]
(-4.5,-0.3)circle(1pt);
\fill[black]
(-3.5,-0.3)circle(1pt);
\fill[black]
(-4.5,-1.3)circle(1pt);
\fill[black]
(-3.5,-1.3)circle(1pt);
\fill[black]
(-4.5,-2.3)circle(1pt);
\fill[black]
(-3.5,-2.3)circle(1pt);
\fill[black]
(-4.5,-3.3)circle(1pt);
\fill[black]
(-3.5,-3.3)circle(1pt);
\fill[black]
(-2.5,-0.3)circle(1pt);
\fill[black]
(-1.5,-0.3)circle(1pt);
\fill[black]
(-2.5,-1.3)circle(1pt);
\fill[black]
(-1.5,-1.3)circle(1pt);
\fill[black]
(-2.5,-2.3)circle(1pt);
\fill[black]
(-1.5,-2.3)circle(1pt);
\fill[black]
(-2.5,-3.3)circle(1pt);
\fill[black]
(-1.5,-3.3)circle(1pt);

\path[font=\normalsize]
(-4,-0.4)node[above]{$I_1$}
(-4.5,-0.6)node[left]{$0$}
(-3.5,-0.6)node[right]{$1$}
(-4,-1.4)node[above]{$I_2$}
(-4.5,-1.6)node[left]{$0$}
(-3.5,-1.6)node[right]{$1$}
(-4,-2.4)node[above]{$I_3$}
(-4.5,-2.6)node[left]{$0$}
(-3.5,-2.6)node[right]{$1$}
(-4,-3.4)node[above]{$I_4$}
(-4.5,-3.6)node[left]{$0$}
(-3.5,-3.6)node[right]{$1$}
(-2,-0.4)node[above]{$I_5$}
(-2.5,-0.6)node[left]{$0$}
(-1.5,-0.6)node[right]{$1$}
(-2,-1.4)node[above]{$I_6$}
(-2.5,-1.6)node[left]{$0$}
(-1.5,-1.6)node[right]{$1$}
(-2,-2.4)node[above]{$I_7$}
(-2.5,-2.6)node[left]{$0$}
(-1.5,-2.6)node[right]{$1$}
(-2,-3.4)node[above]{$I_8$}
(-2.5,-3.6)node[left]{$0$}
(-1.5,-3.6)node[right]{$1$}

(2.1,-2.3)node[left]{$E_6$}
(3.1,-2.3)node[left]{$E_7$}
(2.9,-1.3)node[right]{$E_2$}
(3.9,-1.3)node[right]{$E_3$}

(3.5,-0.8)node[above]{$E_1$}
(2.5,-1.8)node[above]{$E_5$}
(3.5,-1.8)node[below]{$E_4$}
(2.5,-2.8)node[below]{$E_8$};

%(3,-0.8)--(4,-0.8)
%(3,-0.8)--(3,-1.8)
%(4,-0.8)--(4,-1.8)
%(2,-1.8)--(4,-1.8)
%(3,-2.8)--(3,-1.8)
%(2,-2.8)--(2,-1.8)
%(2,-2.8)--(3,-2.8)
\end{tikzpicture}
\caption{On the left an example of a set $E$ and on the right 
the resulting graph $G$ with the identifications:
$(0,1)\sim (0,2)$, $(1,1)\sim (0,3)$, $(1,2)\sim (0,4)\sim(1,5)\sim (0,7)$, $(1,3)\sim(1,4)$,
$(0,5)\sim(0,6)$, $(1,6)\sim(0,8)$, $(1,7)\sim(1,8)$.}
\end{center}
\end{figure}
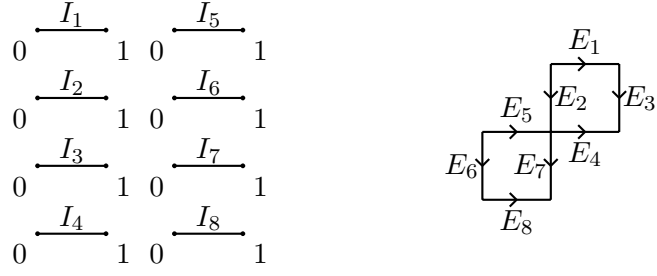

We notice that the natural projection $\pi: E\to G$ 
restricted to $E\setminus V$ is a homeomorphism with its image.

We call
\begin{equation*}
E_i:=I_i\Big{/}\sim\,,\qquad \text{and}\qquad V_G:=V\Big{/}\sim\,.
\end{equation*}

\begin{defn}[$N$--network]\label{network}
An $N$--network (or simply a network) is a pair $\mathcal{N}=(G,\Gamma)$
where
\begin{equation*}
\Gamma: G\to \mathbb{R}^d
\end{equation*}
is a continuous map and $G$ is an $N$--graph.
Moreover we ask each $\gamma^i:=\Gamma_{\vert E_i}$ to be
either a constant map (\emph{singular} curve)  or an immersion of class $H^2$
(\emph{regular} curve).

We will identify two networks $(G_1,\Gamma_1)$ and $(G_2,\Gamma_2)$ if $G_1=G_2=G$ and for any edge $E_i$ of $G$ the curve $\Gamma_1|_{E_i}$ coincides with $\Gamma_2|_{E_i}$ up to reparametrization.
\end{defn}

\begin{defn}[Convergence of networks]
We define that 
a sequence of networks $\mathcal{N}_n= (G_n,\Gamma_n)$ 
converges to a limit network $\mathcal{N}_\infty=(G_\infty, \Gamma_\infty)$ 
in a certain topology if $G_n=G_\infty$ for any $n$ and
each sequence of curves 
$\Gamma_n\vert_{E_i}$
converges to $\Gamma_\infty \vert_{E_i}$, 
up to reparametrization, 
in that topology.
\end{defn}

\begin{defn}[Singular part]\label{def:singular}
Given a network $\N=(G,\Gamma)$ we define its \emph{singular part} $\mathrm{Sing}(\N)$ 
as the subgraph $H\subset G$ whose edges are exactly those associated to singular curves of $\N$ 
(mapped by $\Gamma$ to a point) and we call 
regular part  $\mathrm{Reg}(\N):=\overline{G\setminus\mathrm{Sing}(\N)}$.
\end{defn}

If $E_i\subset\mathrm{Reg}(\mathcal{N})$ has a common vertex $p$ with
$E_j\in\mathrm{Sing}(\mathcal{N})$ we say that $p$ belongs to $\mathrm{Sing}(\mathcal{N})$,
i.e. $\mathrm{Sing}(\mathcal{N})$ is closed.

\begin{defn}[Regular and singular networks]
We say that a network $\N=(G,\Gamma)$ is \emph{regular}
if $\mathrm{Sing}(\N)=\emptyset$ and \emph{singular} otherwise.
\end{defn}

\begin{defn}
Given an network $\mathcal{N}$
we denote by $\ell (\gamma^i)$ the length of the curve $\gamma^i$.
The \emph{length} of the network $\mathcal{N}$ is nothing but
\begin{equation*}
L(\mathcal{N}):=\sum_{i=1}^N \ell (\gamma^i)\,.
\end{equation*}
\end{defn}

Consider a regular curve $\gamma^i$ of 
a  network $\mathcal{N}$.
Then we define its curvature vector as the $L^2$ function given by
\begin{equation*}\label{eq:curvature}
\vec{k}^i=\frac{\ddot\gamma^i}{\vert \dot\gamma^i\vert^2}
-\frac{\left\langle \ddot\gamma^i,\dot\gamma^i
\right\rangle \dot\gamma^i}{\vert \dot\gamma^i\vert^4}\,.
\end{equation*}
Observe that by the Sobolev embedding $\dot\gamma$ is H\"{o}lder continuous, 
and thus by the regularity its norm is bounded and bounded away from zero.

We are now in position to introduce the functional of our interest.

\begin{defn}\label{elasticenergy}
Let $\alpha,\beta >0$.
The \emph{elastic energy functional} $\mathcal{E}_{\alpha,\beta}$
for a regular network $\mathcal{N}$ is defined as 
\begin{equation}\label{eef}
\mathcal{E}_{\alpha,\beta}\left(\mathcal{N}\right):= 
\alpha \int_{\mathcal{N}} \vert \vec{k}\vert^{2}\,\mathrm{d}s
+\beta\, L(\mathcal{N})
=\sum_{i=1}^N\left(\alpha\int_{\mathcal{N}^{i}} \vert \vec{k}^i\vert^{2} \,\mathrm{d}s^i
+\beta \, \ell (\gamma^i) \right)\,,
\end{equation}
where 
$\vec{k}^i$ is the curvature and 
$s^i$ is the arclength 
parameter.
\end{defn}

\begin{rem}\label{reparwithconstspeed}
Reparametrizing the regular curves of the network 
with constant speed equal to the length 
we can write~\eqref{eef} as
\begin{equation}\label{riparametrizzato}
\mathcal{E}(\mathcal{N})=\sum_{i=1}^N\left(
\alpha \int_{0}^1 \frac{\vert \ddot\gamma^i(x)\vert^2 }{\left(\ell(\gamma^i)\right)^3}
\,\mathrm{d}x
+\beta \ell(\gamma^i)\right)\,.
\end{equation}
\end{rem}

By assigning 
to each singular curve of a network
a zero curvature $\vec{k}^i=\vec{0}$
we can naturally extend $\mathcal{E}$
to the class of all networks
maintaining exactly the same formula.
The elastic energy 
$\widetilde{\mathcal{E}}_{\alpha,\beta}$
of a regular or singular network $\mathcal{N}$ is thus
\begin{equation}\label{funzionalegenerale}
\widetilde{\mathcal{E}}_{\alpha,\beta}\left(\mathcal{N}\right):= 
\alpha \int_{\mathcal{N}} \vert \vec{k}\vert^{2}\,\mathrm{d}s
+\beta\, L(\mathcal{N})\,,
\end{equation}
where
$\vec{k}^i$ is the curvature of the regular curves of $\mathcal{N}$
or the assigned zero curvature on the singular curves of $\mathcal{N}$.

%%%%%%%%%%%%%%%%%%%%%%%%%%%%%%%%%%%%%%%%%%%%%%%%%%
\subsection{Definition of the minimization problem}
%%%%%%%%%%%%%%%%%%%%%%%%%%%%%%%%%%%%%%%%%%%%%%%%%%

\begin{defn}[Junction of order $m$]
Consider an $N$--graph $G$ and $p\in V_G$.
We say that $p$ is a junction of order $m$ (with $m\in \{1,\ldots, N\}$) if 
\begin{equation*}
\sharp \,
\pi^{-1} (p)=m\,,
\end{equation*}
where $\pi$ is the projection defined below Definition~\ref{defgraph}
and $\sharp$ denotes the cardinality of a set.
\end{defn}

\begin{defn}[Graph with assigned angles]\label{defTheta}
We say that an $N$--graph $G$ has assigned angles 
if to every junction $p\in V_G$
we assign \emph{directions}
\begin{equation*}
d^{z_1,i_1},\ldots,d^{z_m,i_m}\in\mathbb{S}^{d-1}
\end{equation*}
where $m$ is the order of the junction and where
\begin{equation*}
p=\pi\left(z_{1},i_1\right)=\ldots= 
\pi\left(z_{m},i_m\right)\,,
\end{equation*}
with $(z_{1},{i_1}),\ldots,  (z_{m},i_m)\in \{0,1\}\times\{1,\ldots,N\}$ all distinct.
We denote by $\mathcal{D}$ the set of all the assigned directions.
\end{defn}

\begin{defn}[Angle condition for a regular network]\label{angolifissati}
Given $G$ an $N$--graph with assigned angles, we say that
the regular network $\mathcal{N}=(G,\Gamma)$
fulfills the angle condition if for every 
$p\in V_G$ junction of order $m$, writing 
\begin{equation*}
p=\pi\left(z_{1},i_1\right)=\ldots= 
\pi\left(z_{m},i_m\right)\,,
\end{equation*}
with $(z_{1},{i_1}), \ldots,  (z_{m},i_m)\in \{0,1\}\times\{1,\ldots,N\}$  all distinct,
there exists a rotation $\mathrm{R}_p:\mathbb{R}^d\to\mathbb{R}^d$ depending on $p$ such that
\begin{equation*}
\tau^{z_1,i_1}=\mathrm{R}_p(d^{z_1,i_1})\,,\ldots,\tau^{z_m,i_m}=\mathrm{R}_p(d^{z_m,i_m})\,,
\end{equation*}
where $\tau^{z,i}$ are the outer tangents introduced in
Definition~\ref{outertang}.
\end{defn}

\begin{defn}[Class $\mathcal{C}_{\rm{Reg}}$]
Fix an $N$--graph with assigned angles. We 
say that an network $\mathcal{N}=(G,\Gamma)$
belongs to $\mathcal{C}_{\rm{Reg}}(\mathcal{D})$ if it is regular and fulfills the angle condition
with the directions assigned by $\mathcal{D}$
in the sense of Definition~\ref{angolifissati}.
\end{defn}

\begin{prob}\label{problem}
Given an $N$--graph $G$ with assigned angles by $\mathcal{D}$
we want to study 
\begin{equation}\label{inf}
\inf 
\left\lbrace \mathcal{E}_{\alpha,\beta}(\mathcal{N})\;\vert\: \mathcal{N}
=(G,\Gamma)\in\mathcal{C}_{\rm{Reg}(\mathcal{D})} 
\right\rbrace \,. 
\end{equation}
\end{prob}

From now on, for sake of notation, we simply write $\mathcal{C}_{\rm{Reg}}$ instead of
$\mathcal{C}_{\rm{Reg}}(\mathcal{D})$.

\begin{rem}
It is not restrictive to ask $G$ to be connected, otherwise
one minimizes the energy of each connected component of $G$.
\end{rem}

\begin{rem}
Let us call $\mathcal{E}:=\mathcal{E}_{1,1}$.
The rescaling properties of the functional $\mathcal{E}_{\alpha,\beta}$ imply
\begin{equation}\label{eq:scaling}
\mathcal{E}_{\alpha,\beta}(\mathcal{N})=
\sqrt{\alpha\beta}\,\mathcal{E}\left(\sqrt{\frac{\beta}{\alpha}}\,\mathcal{N}\right)\,,
\end{equation}
so if $\mathcal{N}_{\min}$ is a minimizer for $\mathcal{E}$, then
the rescaled network  $\frac{\beta}{\alpha}\,\mathcal{N}_{\min}$
is a minimizer for $\mathcal{E}_{\alpha,\beta}$ and vice versa.
Hence it is not restrictive to fix $\alpha=\beta=1$.
\end{rem}

\begin{rem}
Given a graph  $G$ with assigned angles, there 
always exists a map $\Gamma:G\to\mathbb{R}^d$ such that 
$(G,\Gamma)$ is a regular network with finite energy.
It is sufficient to send each vertex $p\in V_G$ to a point $x^j$ in $\R^d$
and to connect the points $x^i$ 
with curves of finite length and with bounded curvature
whose outward tangent vectors are chosen accordingly 
with the fixed directions $d^{z_j,i_j}$.
The class $\mathcal{C}_{\rm{Reg}}$ is hence not empty and~\eqref{inf} is finite.
\end{rem}

%%%%%%%%%%%%%%%%%%%%%%%%%%%%%%%%%%%%%%%%%%%%%%%%%%
%%%%%%%%%%%%%%%%%%%%%%%%%%%%%%%%%%%%%%%%%%%%%%%%%%
\section{Compactness}\label{cptn}
%%%%%%%%%%%%%%%%%%%%%%%%%%%%%%%%%%%%%%%%%%%%%%%%%%
%%%%%%%%%%%%%%%%%%%%%%%%%%%%%%%%%%%%%%%%%%%%%%%%%%

With a little abuse of notation
by considering a sequence of networks $\{\mathcal{N}_n\}_{n\in\mathbb{N}}$
we mean that we consider a sequence of pairs $(G,\{\Gamma_n\}_{n\in\mathbb{N}})$
where the $N$--graph $G$ with assigned angles is fixed.

\begin{lem}\label{pocaoscillazione}
Let $\{\mathcal{N}_n\}_{n\in\mathbb{N}}$
be a sequence of networks in $\mathcal{C}_{\mathrm{Reg}}$
such that
\begin{equation*}
\limsup_n \mathcal{E}(\mathcal{N}_n)\leq C<+\infty\,.
\end{equation*}
Suppose that for a certain index $i\in\{1,\ldots, N\}$
\begin{equation*}
\lim_{n\to\infty}\ell(\gamma^i_n)=0\,,
\end{equation*}
then
%for all $x,y \in [0,1]$ it holds
\begin{equation*}
\lim_{n\to\infty}\,\sup_{x,y\in [0,1]} \left\lvert 
\frac{\dot\gamma^i_n(x)}{\vert\dot\gamma^i_n(x)\vert}
-\frac{\dot\gamma^i_n(y)}{\vert\dot\gamma^i_n(y)\vert}
\right\rvert=0\,.
\end{equation*}
\end{lem}
\begin{proof}
Since $\{\mathcal{N}_n\}_{n\in\mathbb{N}}$ is a sequence of regular networks,
we can suppose (up to reparametrisation)
that for every $n\in\mathbb{N}$, for every $i\in\{1,\ldots,N\}$ the immersion $\gamma^i_n:[0,1]\to\mathbb{R}^2$
is a regular parametrization with constant speed equal to its length.
Given $x,y\in[0,1]$, we get
\begin{align*}
\left\lvert 
\frac{\dot\gamma^i_n(x)}{\vert\dot\gamma^i_n(x)\vert}
-\frac{\dot\gamma^i_n(y)}{\vert\dot\gamma^i_n(y)\vert}
\right\rvert
&  
=\frac{1}{\ell(\gamma^i_n)} \left|  \dot\gamma^i_{n}(x)-\dot\gamma^i_{n}(y) \right|
=\frac{1}{\ell(\gamma^i_n)} \left| \int_x^y \ddot\gamma^i_{n}(t)\frac{\ell(\gamma^i_n)^2}{\ell(\gamma^i_n)^2}\,\mathrm{d}t\, \right| \\
& \leq  \left| \int_{\gamma^i_n} |\vec{k}^i_n| \, ds \right| 
\leq \mathcal{E}(\mathcal{N}_n)^{1/2} \ell(\gamma^i_n)^{1/2} \\
&\leq C \ell(\gamma^i_n)^{1/2}\,.
\end{align*}
We then obtain the desired result passing to the limit.
\end{proof}

%{\color{red}
%Ritorniamo su questa dimostrazione del lemma 3.1. per accontenare tutti i referee.
%Forse una versione che piace sia al referee 2 che 4 potrebbe essere:

%Given $x,y\in[0,1]$, we get
%\begin{align*}
%\left\lvert 
%\frac{\dot\gamma^i_n(x)}{\vert\dot\gamma^i_n(x)\vert}
%-\frac{\dot\gamma^i_n(y)}{\vert\dot\gamma^i_n(y)\vert}
%\right\rvert
%&  
%=\frac{1}{\ell(\gamma^i_n)} \left|  \dot\gamma^i_{n}(x)-\dot\gamma^i_{n}(y) \right|
%=\frac{1}{\ell(\gamma^i_n)} \left| \int_x^y \ddot\gamma^i_{n}(t)\,\mathrm{d}t\, \right| \\
%& \leq  \frac{1}{\ell(\gamma^i_n)}
%\left( \int_x^y  \ddot\gamma^i_{n}(t) ^2\,\mathrm{d}t\right)^{1/2}  
%\left( \int_x^y 1\,\mathrm{d}t\right)^{1/2}\\
%&= \left( \int_x^y  \frac{\ddot\gamma^i_{n}(t)^2}
%{\ell(\gamma^i_n)^3} \,\mathrm{d}t\right)^{1/2}
%\ell(\gamma^i_n)^{1/2} \left( \int_x^y 1\,\mathrm{d}t\right)^{1/2}\\
%& \leq \mathcal{E}(\mathcal{N}_n)^{1/2} \ell(\gamma^i_n)^{1/2}
%\leq C \ell(\gamma^i_n)^{1/2}\,,
%\end{align*}
%where in the last line we use the expression of $\mathcal{E}$ presented in 
%Remark~\ref{reparwithconstspeed}.
%We obtain the desired result passing to the limit.
%}
We state here again the angle condition~\ref{angcondshort} with the 
use of the notation introduced in Section~\ref{definitionoftheproblem}.

\begin{defn}[Angle condition for a singular network]\label{angolifissatiRd}
Consider  a (possibly singular) network $\mathcal{N}=(G, \Gamma)$.
We recall that for any $p\in V_G$ and $\gamma^{i}$ regular curve such that $\pi(z,i)=p$ for some $z\in\{0,1\}$ the usual (real) outward tangent vector is
$\tau^{z,i}=(-1)^{z}\frac{\dot \gamma^{i}(z)}{|\dot \gamma^{i}(z)|}$.

We say that $\mathcal{N}$ satisfies the angle condition (given by a set of directions $\mathcal{D}$) if for every singular curve $\gamma^i$ there exist 
unit vectors $\tau^{0,i}$ and $\tau^{1,i}$ (called virtual tangent vectors) such that $\tau^{0,i}=-\tau^{1,i}$ 
and such that for every $p\in V_G$ there exists a rotation $\mathrm{R}_p$ in $\mathbb{R}^d$ such that
\begin{equation}\label{eq4}
\tau^{z_1,i_1}=\mathrm{R}_p(d^{z_1,i_1}),\,...\,,\tau^{z_m,i_m}=\mathrm{R}_p(d^{z_m,i_m})\,,
\end{equation}
where $\pi^{-1}(p)=\{(z_1,i_1),\,...\,,(z_m,i_m)\}$ and $\tau^{z_j,i_j}$ 
are either the real outward tangent vectors  or the virtual ones.
\end{defn}

\begin{rem}[{Gauss--Bonnet-type theorems for curves in the plane}]\label{rem:GaussBonnet}
Let $\sigma:[a,b]\to\R^2$ be an embedded piecewise $H^2$, regular closed curve, that is, there exist points $a=x_0<x_1< \ldots <x_{n-1}<x_n=b$ such that $\sigma|_{[x_j,x_{j+1}]}$ is a regular curve of class $H^2$, $\sigma$ is continuous on $[a,b]$, and $\sigma(a)=\sigma(b)$. Denote by $\dot\sigma(x^{\pm}_j):=\lim_{t\to x^{\pm}_j} \dot\sigma(t)$, which exists as $\dot\sigma(t) $ is $\tfrac12$-H\"{o}lder continuous on $[x_j,x_{j+1}]$ for $j=0,\ldots, n-1$. Finally denote by $\Theta[v,w] \in (-\pi,\pi] $ the angle between two vectors $v,w \in \R^2$, taken with positive sign if and only if the ordered couple $(v,w)$ is a positive basis of $\R^2$.
Then the classical Gauss--Bonnet theorem (which follows for instance from Hopf's Umlaufsatz \cite[Theorem 2.4.7]{AbToCurveSuperfici}) reads
\[
\int_a^b \langle \vec{k}_\sigma , \nu_\sigma \rangle \,  = 2\pi - \sum_{j=0}^{n-1} \Theta[\dot\sigma (x^-_j), \dot \sigma(x^+_j) ],
\]
where $\vec{k}_\sigma$ is the curvature vector of $\sigma$, $\nu_\sigma$ is its normal vector, $\dot \sigma(x_0^-) \equiv \dot \sigma(a^-):= \dot \sigma(b^-)$, and we are assuming that $\sigma$ positively parametrizes the boundary of the bounded connected component it encloses.
% \[
% \int_a^b |\vec{k}_\sigma|\, ds \ge 2\pi - \left| \sum_{j=0}^{n-1} \Theta[\dot\sigma (x^-_j), \dot \sigma(x^+_j) ] \right|,
% \]
% where $\vec{k}_\sigma$ is the curvature vector of $\sigma$ and $\dot \sigma(x_0^-) \equiv \dot \sigma(a^-):= \dot \sigma(b^-)$.
% \red{G: Rileggendo, terrei solo una disuguaglianza, oppure toglierei i moduli dalla prima.}

Below we will need a refined version of this theorem. Suppose now that the curve $\sigma$ has the same regularity as above but it is possibly immersed, i.e., it may have self--intersections, then the Gauss--Bonnet theorem (see \cite[Theorem A.1]{danovplu}) reads
\begin{equation}\label{eq:GaussBonnetImmerso}
    \int_a^b |\vec{k}_\sigma|\, ds \ge 2\pi -  \sum_{j=0}^{n-1} \left|\Theta[\dot\sigma (x^-_j), \dot \sigma(x^+_j) ] \right|.
\end{equation}
\end{rem}

We discuss here more in detail the example, anticipated in the Introduction, of a singular network that satisfies the angle 
condition~\ref{angolifissatiRd} but that cannot be the limit of a sequence of regular networks
with uniformly bounded energy.

\begin{Example}\label{condnecessarianonsuff}
Let the dimension of the ambient space be $d=2$.
Consider the $6$--graph of Figure~\ref{cicloproblematico}
where in particular 
the edges $E_1,E_2,E_3,E_4$ form a cycle $C$
\footnote{Namely the equivalence relation that defined $G$
is given by $(1,i)\sim (0,i+1)$ with the index $i$ considered modulo $4$.}
and, as sketched in the picture, assume that the assigned directions $\{d^{z,i}\}$ at the (triple) junctions form (equal) angles equal to $\tfrac{2}{3}\pi$.
% {\color{red} SCRIVERE MEGLIO QUI}
% \begin{align*}
% \angle d^{1,1},d^{0,2}=\angle d^{1,3},d^{0,4}=\tfrac{2\pi}{3}\,,
% \angle d^{1,2},d^{0,3}=\angle d^{1,4},d^{0,1}=-\tfrac{2\pi}{3}\,.
% \end{align*}
% \begin{figure}[H]
% \begin{center}
% \begin{tikzpicture}[scale=0.8]
% \draw[color=red, thick]
% (1,-0.86)to[out= 0,in=0](-1.5,2)
% (-1.5,2)to[out= 180,in=180](-1,0.86)
% (-1,0.86)--(-0.5,0)
% (-0.5,0)--(0.5,0)
% (0.5,0)--(1,-0.86);
% \draw[ thick]
% (-0.5,1.72)to[out= 60,in=0] (-1.3,2.75)
% (-1.3,2.75)to[out= 180,in=-120] (-1,-0.86)
% (-1,0.86)--(-0.5,1.72)
% (-0.5,0)--(-1,-0.86);
% \draw[rotate=180,  thick]
% (-0.5,1.72)to[out= 60,in=0] (-1.3,2.75)
% (-1.3,2.75)to[out= 180,in=-120] (-1,-0.86)
% (-1,0.86)--(-0.5,1.72)
% (-0.5,0)--(-1,-0.86);
% \path[font=\normalsize]
% (-0.7,0.5)node[left]{$E_2$}
% (-0.5,0.3)node[right]{$E_3$}
% (1.12,-0.8)node[above]{$E_4$}
% (1,1.5)node[left]{$E_1$};
% \end{tikzpicture}
% \caption{A regular network $\N=(G,\Gamma)$ 
% satisifying the angle condtition defined above.}\label{figura:angoli}
% \end{center}
% \end{figure}
%
Consider a sequence of continuous maps $\{\Gamma_n\}_{n\in\mathbb{N}}$
such that $\Gamma_1$ is, for example, as depicted in Figure~\ref{cicloproblematico} and such that $L(C_n)\to0$, where $C_n:={\Gamma_n}\vert_{C}$.
This property is satisfied by any sequence of networks approximating the singular network $\N_\infty=(G,\Gamma_\infty)$ depicted on the right in Figure~\ref{cicloproblematico}, which collapses the cycle $C$. Observe that there exists a choice of virtual tangents for which $\N_{\infty}$ satisfies the angle condition in the sense of Definition~\ref{angolifissatiRd}.
But combining \eqref{eq:GaussBonnetImmerso} and H\"{o}lder inequality, we get
% \begin{equation*}
% \mathcal{E}(\Gamma_n)\geq
% \mathcal{E}(C_n)\geq \int_{C_n}\vert \vec{k}\vert^2 \,\mathrm{d}s
% \geq \frac{\left(\int_{C_n}\vert \vec{k}\vert \,\mathrm{d}s\right)^2}{L(C_n)}
% \geq \frac{4\pi^2}{ L(C_n)}\,,
% \end{equation*}
\begin{equation*}
\mathcal{E}(\Gamma_n)\geq
\mathcal{E}(C_n)\geq \int_{C_n}\vert \vec{k}\vert^2 \,\mathrm{d}s
\geq \frac{\left(\int_{C_n}\vert \vec{k}\vert \,\mathrm{d}s\right)^2}{L(C_n)}
\geq \frac{\left( 2\pi - 4\frac\pi3 \right)^2}{ L(C_n)}\,,
\end{equation*}
hence, as $n\to\infty$, the energy diverges as the length of the red loop $C_n$ goes to zero.
% This fact is true not only for the network
% depicted in Figure~\ref{cicloproblematico}
% but for any network $(G,\Gamma_n)$ even if 
% $\Gamma_n$ is merely a sequence of 
% immersions rather than of embeddings. Indeed using
% Remark~\ref{rem:GaussBonnet} we get
% \begin{equation*}
% \mathcal{E}(\Gamma_n)\geq
% \mathcal{E}(C_n)\geq \int_{C_n}\vert \vec{k}\vert^2 \,\mathrm{d}s
% \geq \frac{\left(\int_{C_n}\vert \vec{k}\vert \,\mathrm{d}s\right)^2}{L(C_n)}
% \geq \frac{\left( 2\pi - 4\frac\pi3 \right)^2}{ L(C_n)}\,,
% \end{equation*}
% and so the energy will always 
% diverge if the length  goes to zero.
\end{Example}

The angle condition for a singular network is thus not sufficient to 
characterize the set of limit networks with bounded energy.

\begin{defn}[Straight graph]
An $N$--graph $G$ with assigned angles is \emph{straight} 
if there exists a regular network $\N=(G,\Gamma)$ 
whose curves are 
%non-degenerate 
straight segments that fulfills the angle condition in the sense of
Definition~\ref{angolifissati}.
\end{defn}

Suppose that $G$ is an $N$--graph with assigned angles.
Then every subgraph $H$ of $G$ inherits the assigned angles from $G$
in the sense that to every 
vertex  $p=\pi(z,i)$ of an edge $E_i$ in $H$
we assign $d^{z,i}\in\mathbb{S}^{d-1}$
coinciding with the assignment of $G$.\\

\noindent For the convenience of the reader, let us recall here the key notion of stratified straight subgraph, which was already presented in Definition \ref{shortstratstraightRd}.

\begin{defn}[Stratified straight subgraph]\label{def:stratstraightRd}
Given a graph $G$ with assigned angles
we say that a subgraph $H\subseteq G$ is stratified--straight if there exists a finite sequence of subgraphs,
called \emph{strata},
\begin{equation*}
\emptyset=H_q\subset H_{q-1}\subset\ldots\subset H_1\subset H_0=H
\end{equation*}
and maps $\Sigma_j:H_j\to \R^d$  
%such that $(H_0,\Sigma_0)$ is a 
%(possibly singular) network satisfying the angle condition 
%in the sense of Definition~\ref{angolifissatiRd} 
%%with (real or fictitious) tangents $\tau^{z,i}$ 
%and 
such that for $j=0,\ldots,q-1$ we have that
\begin{itemize}
\item the (sub)network  $(H_j,\Sigma_j)$ satisfies the angle condition
in the sense of Definition~\ref{angolifissatiRd} 
with (real or virtual) tangent vectors coinciding with the ones associated to $(H_0,\Sigma_0)$ 
and whose regular curves are straight segments;
\item $H_{j+1}=\mathrm{Sing}((H_j,\Sigma_j))$.
\end{itemize}
We call \emph{step} of $G$ the least $q$ for which the above holds.
\end{defn}

Every straight graph with assigned angles is trivially stratified-straight, 
but the converse does not hold in general as shown by the following example.

\begin{Example}\label{strstr}
We consider the graph of Figure~\ref{fig:twostrata}
characterized by the following identifications:
\begin{align*}
&(0,1)\sim (0,2), \,\,
(1,2)\sim (1,3)\sim (0,5), \,\,
(1,1)\sim (0,3)\sim (0,4), \,\,
(1,4)\sim (1,5)\,,\\
&\measuredangle d^{0,2},d^{0,1}=\measuredangle d^{0,3},d^{0,4}=0\,,\, \,\,\quad
\measuredangle d^{1,1},d^{0,3}=\measuredangle d^{1,3},d^{1,2}= \measuredangle d^{1,4},d^{1,5}=\tfrac{\pi}{2}\,,
\end{align*}
where $\measuredangle v,w \in[0,2\pi)$ here identifies the least positive angle between two vectors $v,w$ in the plane (see Figure \ref{fig:twostrata} on the left).

It is clearly not possible to construct a triangle with strictly positive length of its three (straight) edges with angles $0,\tfrac{\pi}{2},\tfrac{\pi}{2}$.
The minimal step of $G$ seen as a stratified straight graph is then $2$,
with $H_0=G$, $H_1=E_3\cup E_4\cup E_5$ and $H_2=E_5$.

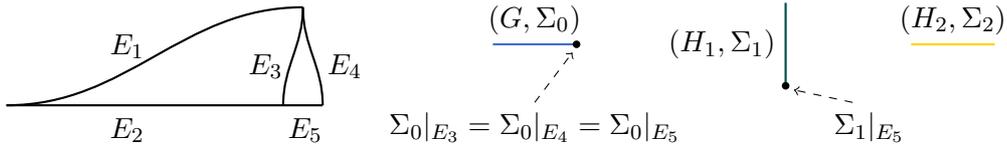
\begin{figure}[H]
\centering
\begin{tikzpicture}[scale=1.3]
\begin{scope}[rotate=-90, thick]
\draw (1,0.2) -- (1,-3);
%\draw (1,3) to [out=-90,in=90] (0,0);
\draw (1,-3) to [out=90,in=-90] (0,0);
\draw (0,0) to [out=0,in=180] (1,0.2);
\draw (0,0) to [out=0,in=180] (1,-0.2);
\end{scope}
\path[font=\normalsize]
(-1.5,-0.4)node[left]{$E_1$}
(-1.5,-1.25)node[left]{$E_2$}
(-0.1,-0.6)node[left]{$E_3$}
(0.7,-0.6)node[left]{$E_4$}
(0.3,-1.25)node[left]{$E_5$};
\end{tikzpicture}
\begin{tikzpicture}[scale=1.1]
%stratum H_0\H_1
\draw[hanblue, thick]
(-1,0.5)--(0,0.5);
\draw[->,dashed] (-0.5,-0.2)--(-0.05,0.4);
%stratum H_1\H_2
\draw[deepjunglegreen, thick]
(2.5,0)--(2.5,1);
\draw[->,dashed] (3.3,-0.2)--(2.6,-0.05);
%stratum H_2\H_3
\draw[yellow(ncs), thick]
(4,0.5)--(5,0.5);
\fill (0,0.5)
circle(1.3pt); 
\fill (2.5,0)
circle(1.3pt); 
\path[font==\normalsize]
(-0.5,0.5)node[above]{$(G,\Sigma_0)$}
(-0.5,-0.2)node[below]{$\Sigma_0\vert_{E_3}=\Sigma_0\vert_{ E_4}=\Sigma_0\vert_{ E_5}$}
(2.5,0.5)node[left]{$(H_1,\Sigma_1)$}
(3.5,-0.2)node[below]{$\Sigma_1\vert_{ E_5}$}
(4.5,0.5)node[above]{$(H_2,\Sigma_2)$};
\end{tikzpicture}
\caption{Example of stratified straight graph of step two.}\label{fig:twostrata}
\end{figure}
\end{Example}

Notice that the notion of straight and stratified straight graphs coincide when 
the underlying graph $G$ has the structure of a tree with no cycles.
In Section~\ref{strstrcoinceideconstr} we introduce a class 
of graphs possibly with cycles for which every stratified straight
subgraph is also straight.

\begin{defn}[Degenerate $N$--network]\label{degnetwRd}
Let $G$ be an $N$--graph with assigned angles. 
A network $\mathcal{N}=(G,\Gamma)$
is degenerate if
\begin{itemize}
\item it satisfies the angle condition in the sense of Definition~\ref{angolifissatiRd}, with (real or virtual) tangents $\tau^{z,i}$;
\item its singular part $\mathrm{Sing}(\N)$ is a stratified-straight subgraph with (real or virtual) tangents coinciding with the $\tau^{z,i}$ above.
\end{itemize}
\end{defn}

We denote by $\mathcal{C}_{\rm{Deg}}$ 
the  class of  degenerate networks.
We remark that 
by definition $\mathcal{C}_{\rm{Reg}}\subset\mathcal{C}_{\rm{Deg}}$.
We remind that to compute the elastic energy of a network
in $\mathcal{C}_{\rm{Deg}}$  we make use of the
extension of the functional defined in~\eqref{funzionalegenerale}.

\begin{prop}[Compactness]\label{regolariconvergonoadeg}
Let $\{\mathcal{N}_n\}_{n\in\mathbb{N}}$ be a sequence of 
networks in  $\mathcal{C}_{\rm{Deg}}$ such that
$$
\limsup_n \widetilde{\mathcal{E}}(\mathcal{N}_n)\leq C<+\infty\,.
$$
Then $\mathcal{N}_n$ converges (up to subsequence and translation)
to $\mathcal{N}_\infty\in \mathcal{C}_{\rm{Deg}}$
weakly in $H^2$ and strongly in $C^{1, \alpha}$ for every $\alpha \in (0,\nicefrac{1}{2})$.
In particular if $\{\mathcal{N}_n\}_{n\in\mathbb{N}}$ is a sequence of networks in 
$\mathcal{C}_{\rm{Reg}}$, then 
$\mathcal{N}_n\overset{H^2}{\rightharpoonup}\mathcal{N}_\infty\in \mathcal{C}_{\rm{Deg}}$.
\end{prop}

\begin{proof}
Up to subsequence and up to relabeling the edges $E_i$ we can suppose that 
for a certain $k\in\{1,\ldots N\}$ for every $n\in\mathbb{N}$
the curves $\mathcal{N}_n^1,\ldots,\mathcal{N}_n^k$ are regular 
and the curves $\mathcal{N}_n^{k+1},\ldots,\mathcal{N}_n^N$
are singular.
Without loss of generality we suppose that for every $n\in\mathbb{N}$
and for every $i\in\{1,\ldots,k\}$
the immersions $\gamma^i_n:=\Gamma_{n}\vert_{{E_i}}:[0,1]\to\mathbb{R}^d$
are regular parametrizations with constant speed equal to the length and
we ask that $\gamma_n^1(0)$ is mapped into the origin of $\mathbb{R}^d$.
For every $i\in\{1,\ldots,k\}$ and for every $n\in\mathbb{N}$ we have that
\begin{align}
\| \gamma^i_n\|_\infty&=
\sup_{x\in [0,1]}\vert {\gamma}_n^i(x)\vert
\leq L(\mathcal{N}_n)\leq \widetilde{\mathcal{E}}(\mathcal{N}_n)\leq C<\infty\,,\label{normafunz}\\
\Vert \dot\gamma^i_n\Vert_\infty&=
\sup_{x\in [0,1]}\vert \dot{\gamma}_n^i(x)\vert
=\ell (\mathcal{N}_n^i)\leq L(\mathcal{N}_n)\leq \widetilde{\mathcal{E}}(\mathcal{N}_n)\leq C<\infty\,,
\label{normader} \\
\Vert \ddot\gamma^i_n\Vert_2^2&=\ell (\mathcal{N}_n^i)^3
\int_0^1 \frac{\vert\ddot\gamma^i_n(x)\vert^2}{\ell (\mathcal{N}_n^i)^3}\,\mathrm{d}x
\leq L(\mathcal{N}_n)^3\widetilde{\mathcal{E}}(\mathcal{N}_n)\leq C^4<\infty\,.\label{normadersec}
\end{align}

Since for every $i\in\{k+1,\ldots,N\}$ 
the maps are constant on the bounded interval of parametri--zation $[0,1]$,
a bound on $\Vert\cdot\Vert_{H^2}$ is trivially obtained.
%Then for every $i\in\{1,\ldots,N\}$
%up to a subsequence (not relabeled)
%$\gamma^i_{n}\rightharpoonup \gamma^i_\infty$ weakly in $H^2(0,1)$ and
%$\gamma^i_{n}\rightarrow \gamma^i_\infty$
%strongly in $C^{1,\alpha}([0,1])$ for every $\alpha\in(0,\tfrac12)$.
Then for every $i\in\{1,\ldots,N\}$
up to a subsequence (not relabeled)
$\gamma^i_{n}\rightharpoonup \gamma^i_\infty$ weakly in $H^2(0,1)$ and
thanks to classical compact embedding theorems $\gamma^i_{n}\rightarrow \gamma^i_\infty$
strongly in $C^{1,\alpha}([0,1])$ for every $\alpha\in(0,\tfrac12)$.

\medskip

Since all the networks of the sequence $\{\mathcal{N}_n\}_{n\in\mathbb{N}}$
satisfy the angle condition of Definition~\ref{angolifissatiRd}, 
for each $n$ there exists a family of unit vectors $\Tau_n=(\tau_n^{z,i})_{(z,i)\in\pi^{-1}(V_G)}$
that are either the real or the virtual tangents.
By compactness of $\mathbb{S}^{d-1}$
up to subsequence (not relabeled) $\Tau_n$ converge to a limit $\Tau_\infty$ 
composed of elements denoted by $\tau^{z,i}_{\infty}$.
Notice that for the the indices $i$
for which $\ell(\gamma^i_{\infty})>0$
we have that (up to subsequence)
$(-1)^{z}\frac{\dot{\gamma}^i_n}{\vert \dot{\gamma}^i_n\vert}$
converge to $\tau^{z,i}_\infty$.

\medskip

Define $\Gamma_\infty:G\to\mathbb{R}^2$ in such a way that 
$\Gamma_{\infty}\vert_{E_i}=\gamma^i_{\infty}$
and call $\mathcal{N}_{\infty}:=(G,\Gamma_\infty)$.
We want to prove that  $\mathcal{N}_\infty\in \mathcal{C}_{\rm{Deg}}$.

\medskip

First of all we claim that the angle condition is verified for $\N_\infty$ with the family $\Tau_\infty$. 
It is easy to check that there exist rotations $\mathrm{R}_p$ (obtained as limits of those for $\N_n$) 
that verify~\eqref{eq4}. It remains to prove that for the regular curves $\gamma_\infty^i$ the vectors 
$\tau_\infty^i$ coincide with the outer tangents and that for the constant curves we have 
$\tau_\infty^{0,i}=-\tau_\infty^{1,i}$. For each $i\in\{1,\ldots, N\}$ there are three possible cases: a regular 
curve converges to a regular curve; a regular curve converges to a constant curve; a constant curve 
converges to a constant curve.
In the first case the claim follows from the above argument. In the second case the 
claim follows by Lemma~\ref{pocaoscillazione}. In the third case the claim follows trivially.

\medskip

It remains to verify that $H=\Sing(\N_\infty)$ is stratified-straight.
We proceed by induction: we set $H_0=H$ and supposing to have obtained $H_i$ 
we construct $\Sigma_i$ and $H_{i+1}$. We can assume without loss of generality that $H_i$ is connected (otherwise we apply the same argument to each connected component) and up to a translation that $0\in \Gamma_n(H_i)$. We define $L_n^i:=L(\Gamma_n(H_i))$ and consider the rescaled networks $\N_n^{(i)}=(H_i,\Gamma_n^{(i)})$ with
\begin{equation*}
\Gamma_n^{(i)}=\frac{1}{L_n^i}\Gamma_n\,.
\end{equation*}
Then $L(\N_n^{(i)})=1$. Moreover, at each step the network $\mathcal{N}_n^{(i)}$ is defined starting from the singular part of the previous one (indeed $H_0=H$, while for $H_i$, $i\geq 1$, see \eqref{eq:inductionstep} below), therefore we have that $L_n^i\to 0$.  Thanks to the scaling property \eqref{eq:scaling} we thus have
\begin{equation*}
\int_{\N_n^{(i)}}\vert \vec{k}^i\vert^2\,\mathrm{d}s
=L_n^i 
\int_{\Gamma_n(H_i)}\vert\vec{k}^i\vert^2\,\mathrm{d}s\to 0\quad\text{as}\; n\to\infty
\end{equation*}
and in particular $\widetilde{\mathcal{E}}(\N_n^{(i)})\leq C<\infty$ for some constant $C$.
Hence we can repeat the previous reasoning to conclude that 
up to subsequence the networks $\N_n^{(i)}$ converge weakly in $H^2$ and 
strongly in $C^1$ to a network 
$\N_\infty^{(i)}=(H_i,\Gamma_\infty^{(i)})$. By $C^1$--convergence $L(\N_\infty^{(i)})=1$ and thus 
the limit network has at least one regular curve. 
By Lemma~\ref{pocaoscillazione} we have that the curves of $\N_\infty^{(i)}$ are either constant or straight segments, 
since all surviving curves have a constant tangent. 
We then define 
\begin{equation}\label{eq:inductionstep}
    \Sigma_i=\Gamma_\infty^{(i)}\qquad H_{i+1}=\mathrm{Sing}((H_i,\Sigma_i))
\end{equation}
and we repeat the process until $\mathrm{Sing}((H_i,\Sigma_i))=\emptyset$.

After a finite number of steps the process stops and we obtain that the strata 
$\emptyset=H_q\subset\ldots\subset H_0=H$ by construction satisfy 
the condition of Definition~\ref{def:stratstraightRd}.

We have thus concluded that $\N_\infty$ satisfies the angle condition with (real or virtual) tangents $\Tau_\infty$ and $\Sing(\N_\infty)$ 
is stratified--straight whose (real or virtual) tangents are consistent with $\Tau_\infty$, and thus by Definition \ref{degnetwRd}
$\mathcal{N}_{\infty}\in \mathcal{C}_{\rm{Deg}}$.
\end{proof}

To convince the reader that in Definition~\ref{degnetwRd}
we cannot replace the notion of stratified straight subgraph
by the simpler one of straight subgraphs,
and hence the definition of degenerate networks is in a sense sharp,
we give an example of a sequence of regular networks with equibounded 
energy which converges to a singular network
whose singular part is stratified straight but nor straight.

\begin{Example}\label{limitstrstr}
Consider a $5$--graph $G$ where in particular
\[
(0,1)\sim (0,2)\,,\,\,(1,1)\sim (0,3)\,,\,\,(1,2)\sim (1,3)\,,
\]
\[
\measuredangle d^{0,1},d^{0,2}=0\,,\qquad \measuredangle d^{1,1},d^{0,3}=\measuredangle d^{1,2},d^{1,3}=\tfrac{\pi}{2}\,,
\]
where $\measuredangle v,w \in[0,2\pi)$ here identifies the least positive angle between two vectors $v,w$ in the plane (see Figure \ref{fig:LimitStrStr} on the left).

We construct a sequence of regular networks $\mathcal{N}_n=(G,\Gamma_n)$
that converge as $n\to\infty$ to a degenerate network whose singular part is stratified straight.

\begin{figure}[h]
\begin{center}
\begin{tikzpicture}[scale=1, rotate=90]
\draw[thick]
(-1,0)to[out=-135,in=180](-0.2,-1.5)
(-0.2,-1.5)to[out=0,in=-65](0.86,-0.5)
(-1,0)to[out=135,in=180](-0.2,1.5)
(-0.2,1.5)to[out=0,in=65](0.86,0.5);
\draw[thick, color=black,scale=3.71,domain=2.625: 3.15,
smooth,variable=\t,shift={(-0.26,1.001)},rotate=0]plot({1.*sin(\t r)},
{1.*cos(\t r)}) ; 
\draw[thick, color=black,scale=3.71,domain=0:0.52,
smooth,variable=\t,shift={(-0.26,-1.001)},rotate=0]plot({1.*sin(\t r)},
{1.*cos(\t r)}) ; 
\draw[thick, color=black,scale=1,domain=1.04: 2.09,
smooth,variable=\t,shift={(0,0)},rotate=0]plot({1.*sin(\t r)},
{1.*cos(\t r)}) ; 
\path[font=\normalsize]
(-0.2,0)node[left]{$E_1$}
(-0.2,0)node[right]{$E_2$}
(1,0)node[above]{$E_3$}
(-0.2,1.5)node[left]{$E_4$}
(-0.2,-1.5)node[right]{$E_5$};
\end{tikzpicture}\qquad
\begin{tikzpicture}[scale=1]
\draw[thick, color=black!50!white, scale=1,domain=-3.15: 3.15,
smooth,variable=\t,shift={(0,0)},rotate=0]plot({1.*sin(\t r)},
{1.*cos(\t r)}) ; 
\draw[thick, color=black!50!white,dashed,scale=3.71,domain=2.3:2.5,
smooth,variable=\t,shift={(-0.26,1.001)},rotate=0]plot({1.*sin(\t r)},
{1.*cos(\t r)}) ; 
\draw[thick, color=black!50!white,scale=3.71,domain=2.5:2.625,
smooth,variable=\t,shift={(-0.26,1.001)},rotate=0]plot({1.*sin(\t r)},
{1.*cos(\t r)}) ; 
\draw[thick, color=black!50!white,dashed,scale=3.71,domain=-3.01: -2.8,
smooth,variable=\t,shift={(-0.26,1.001)},rotate=0]plot({1.*sin(\t r)},
{1.*cos(\t r)}) ; 
\draw[thick, color=black!50!white,scale=3.71,domain=-3.15: -3.01,
smooth,variable=\t,shift={(-0.26,1.001)},rotate=0]plot({1.*sin(\t r)},
{1.*cos(\t r)}) ; 
\draw[black!50!white]
(0,-1)--(0,0)
(-1,0)--(-1,3.72)
(-1,3.72)--(0,0)
(-1,3.72)--(0.86,0.5)
(-1,0)--(0,0)
(0,0)--(0.86,0.5)
(0,0)--(0.86,-0.5);
\draw[thick, color=red,scale=3.71,domain=2.625: 3.15,
smooth,variable=\t,shift={(-0.26,1.001)},rotate=0]plot({1.*sin(\t r)},
{1.*cos(\t r)}) ; 
\draw[thick, color=red,scale=3.71,domain=0:0.52,
smooth,variable=\t,shift={(-0.26,-1.001)},rotate=0]plot({1.*sin(\t r)},
{1.*cos(\t r)}) ; 
\draw[thick, color=red,scale=1,domain=1.04: 2.09,
smooth,variable=\t,shift={(0,0)},rotate=0]plot({1.*sin(\t r)},
{1.*cos(\t r)}) ; 
\path[font=\normalsize]
(-0.75,3)node[below]{$\alpha$}
(0.15,-0.3)node[below]{$r$}
(-1,1.2)node[left]{$R$}
(0.85,0)node[left]{$2\alpha$};
\end{tikzpicture}
\caption{Representation of the graph $G$ with assigned angles and construction of the sequence.}\label{fig:LimitStrStr}
\end{center}
\end{figure}
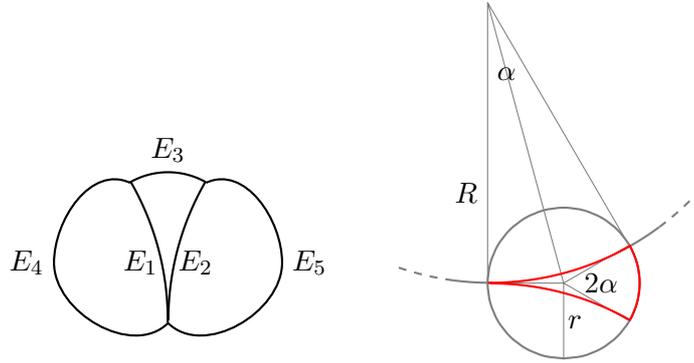

Consider three immersions $\gamma^1,\gamma^2,\gamma^3:[0,1]\to \mathbb{R}^2$
such that
the curve $\gamma^3([0,1])$ is an arc of circle of radius $r$ and length $2\alpha r$
and both $\gamma^1([0,1])$ and $\gamma^2([0,1])$ are arcs of circle of radius $R$ and length
$\alpha R$ with $R=\frac{\sin(\frac{\pi}{2}-\alpha)}{\sin\alpha}r$.
Then
\begin{equation*}
\mathcal{E}(\gamma^1)+\mathcal{E}(\gamma^2)+\mathcal{E}(\gamma^3)
=2\left(r\alpha\cot \alpha
+\frac{\alpha\tan\alpha}{r}\right)+2\alpha r+ \frac{2\alpha}{r}\,.
\end{equation*}
We construct a sequence of immersions 
$\gamma_n^1,\gamma_n^2,\gamma_n^3$ simply
by taking $r_n=\frac{1}{n}$ and $\alpha_n=\frac{1}{n^2}$.
Then as $n\to\infty$ the lengths of the three curves go to zero and
\begin{equation*}
\mathcal{E}(\gamma_n^1)+\mathcal{E}(\gamma_n^2)+\mathcal{E}(\gamma_n^3)
=2\left(\frac{1}{n}
+\frac{1}{n^3}\right)+\frac{2}{n^3}+\frac{2}{n}+o\left(\frac{1}{n^5}\right)
\overset{n\to\infty}{\longrightarrow}0\,.
\end{equation*}
Moreover it is not possible to construct a triangle with three edges of positive length and 
angles $0, \tfrac{\pi}{2},  \tfrac{\pi}{2}$.

Notice that the energy of the sequence need not be infinitesimal.
In fact by taking $r_n=\alpha_n=\frac{1}{n}$, one gets 
that the above energy converges to $2$. 
This highlights the fact that 
curves of a network need not
become flatter and flatter along a sequence even if their length goes to zero.
Hence our 
characterization of the class of degenerate networks
is not related to the curvature of the approximating sequences.

\end{Example}

%%%%%%%%%%%%%%%%%%%%%%%%%%%%%%%%%%%%%%%%%%%%%%%%%%
%%%%%%%%%%%%%%%%%%%%%%%%%%%%%%%%%%%%%%%%%%%%%%%%%%
\section{Relaxation}\label{relax}
%%%%%%%%%%%%%%%%%%%%%%%%%%%%%%%%%%%%%%%%%%%%%%%%%%
%%%%%%%%%%%%%%%%%%%%%%%%%%%%%%%%%%%%%%%%%%%%%%%%%%

Let us define the following extended problem.

\begin{prob}\label{problemarilassato}
Given an $N$--graph $G$ with assigned angles
we want to find
\begin{equation*}\label{infprob}
\inf\left\lbrace \overline{\mathcal{E}}(\mathcal{N})
\;\vert\: \mathcal{N}=(G,\Gamma)\;\text{is a (possibly) singular network}\right\rbrace \,,
\end{equation*}
where $\overline{\mathcal{E}}$ is the extension of $\mathcal{E}$ defined as
\begin{equation}\label{estensione}
\overline{\mathcal{E}}(\mathcal{N}):=
\begin{cases}
\widetilde{\mathcal{E}}(\mathcal{N})
\quad\quad\text{if}\;\mathcal{N}\in \mathcal{C}_{\rm{Deg}}\,,\\
+\infty \;\quad\quad\; \text{otherwise}.
\end{cases}
\end{equation}
\end{prob}

The reason to consider this functional will be clear at the end of this section: we will prove that it coincides 
with the relaxation of the functional $\E$ defined in \eqref{riparametrizzato}, and thus it is the natural 
extension of $\E$ from $\mathcal{C}_{\mathrm{Reg}}$ to $\mathcal{C}_{\mathrm{Deg}}$.
%to the closure of the space of regular networks with finite energy.
\begin{prop}[Recovery sequence]\label{recovery}
For every network $\N=(G,\Gamma)\in\mathcal{C}_{\mathrm{Deg}}$ 
there exists a sequence of networks 
$\{\N_n\}_{n\in\mathbb{N}}$  
such that 
\begin{equation*}
\mathcal{N}_n\overset{H^2}{\longrightarrow}\N\quad
\text{and}\quad\mathcal{E}(\N_n)\to \overline{\mathcal{E}}(\N)\,.
\end{equation*}
\end{prop}

To prove the above proposition we will take advantage of the following preliminary
lemmas.

We recall the following version of a classical lemma.
\begin{lem}\label{lemma:L2dilation}
Let $f\in L^2([0,1];\mathbb{R}^d)$. Then
\[
\|f-f\circ\phi\|_{L^2}\to 0\qquad\text{as}\quad \|\phi-Id\|_{C^0(0,1)}\to 0
\]
where $\phi$ varies among all continuous maps $\phi:[0,1]\to[0,1]$.
%\[
%\lim_{\lambda\to 1^-}\int_0^1|f(t)-f(\lambda t)|^2 dt=0.
%\]
\end{lem}

\begin{proof}
By the triangle inequality and the density of continuous functions on $L^2$ it is sufficient to prove the result when $f$ is continuous. In this case, being defined on a compact set, $f$ is uniformly continuous with some modulus of continuity $\omega_f$. 
Then
\[
\int_0^1|f(t)-f(\phi( t))|^2dt \leq \omega_f(\|\phi-Id\|_{C^0})^2\to 0 \qquad\text{as}\quad \|\phi-Id\|_{C^0}\to 0.
\]
%Then, since $|t-\lambda t|\leq 1-\lambda$ for $t\in [0,1]$, we have
%\[
%\int_0^1|f(t)-f(\lambda t)|^2dt \leq \omega_f(1-\lambda)^2\to 0 \qquad\text{as $\lambda\to 1^-$}.
%\]
\end{proof}

\begin{lem}\label{lem:StimaH2}
Let $l\in(0,+\infty)$ and $0\leq \delta\leq l$. Let $\gamma:[0,1]\to \mathbb{R}^d$ and $\sigma:[1,1+\tfrac{\delta}{l}]\to\mathbb{R}^d$ be two curves of class $H^2$ with $\gamma(1)=\sigma(1)$, $\dot\gamma(1)=\dot\sigma(1)$ and with constant speed $|\dot\gamma|\equiv |\dot\sigma|\equiv l$. Consider the concatenation $\gamma *\sigma:[0,1+\tfrac{\delta}{l}]\to \mathbb{R}^d$ defined in the natural way and its linear reparametrization $\gamma_\sigma:[0,1]\to \mathbb{R}^d$ given by
\[
\gamma_\sigma(t)=(\gamma*\sigma)\left((1+\tfrac{\delta}{l})t\right).
\]
Then
\[
\|\gamma-\gamma_\sigma\|_{H^2(0,1)}\leq \eta\left( \mathcal{E}(\sigma)\right)
\]
where $\eta$ is a function such that $\eta(z)\to 0$ as $z\to 0^+$.
\end{lem}

\begin{proof}
It is sufficient to prove the inequality for the second derivatives, because then by the Poincaré inequality we obtain the inequality for the full $H^2$ norm. We have that:
\begin{align*}
    \int_0^1|\ddot\gamma(t)-\ddot\gamma_\sigma(t)|^2dt  &= \int_0^{\frac{1}{1+\frac{\delta}{l}}}\left|\ddot\gamma(t)-\ddot\gamma\left((1+\tfrac{\delta}{l})t\right)\right|^2dt+\int_{\frac{1}{1+\frac{\delta}{l}}}^1 |\ddot \gamma(t)-\ddot\sigma\left((1+\tfrac{\delta}{l})t\right)|^2dt\\
    &\leq \int_0^{\frac{1}{1+\frac{\delta}{l}}}\left|\ddot\gamma(t)-\ddot\gamma\left((1+\tfrac{\delta}{l})t\right)\right|^2dt+2\int_{\frac{1}{1+\frac{\delta}{l}}}^1 |\ddot \gamma(t)|^2dt\\
    & \phantom{\leq}+2\int_{\frac{1}{1+\frac{\delta}{l}}}^1 |\ddot\sigma\left((1+\tfrac{\delta}{l})t\right)|^2dt.
\end{align*}
Thanks to Lemma \ref{lemma:L2dilation} the first term of the last expression goes to zero as $\delta\to 0$. The second term goes to zero as $\delta\to 0$ by the absolute integrability of $|\ddot\gamma|^2$. Since $\delta=\ell(\sigma)$, then the first two terms go to zero as $\mathcal{E}(\sigma)\to 0$. After the change of variables $x=(1+\tfrac{\delta}{l})t$, the third term reads
\[
\frac{2}{1+\frac{\delta}{l}}\int_1^{1+\frac{\delta}{l}} |\ddot\sigma(x)|^2\,dx.
\]
Since $\sigma$ is parametrized with constant speed $l$, recalling \eqref{eq:curvature} we have
\[
\int_1^{1+\frac{\delta}{l}} |\ddot\sigma(x)|^2\,dx=l^3\int_\sigma |\vec{k}_\sigma|^2 \, ds \leq l^3\mathcal{E}(\sigma).
\]
Therefore all three terms go to zero as $\mathcal{E}(\sigma)\to 0$, and the thesis is proven.
\end{proof}

\begin{rem}\label{rem:StimaH2Concatenation}
We will use the previous lemma in the following way: whenever we have a curve $\gamma$ parametrized with constant speed and we concatenate it with another curve $\sigma$, then the constant speed parametrization on $[0,1]$ of $\gamma*\sigma$ is close in $H^2$ to $\gamma$ if $\mathcal{E}(\sigma)$ is small.
\end{rem}

\begin{lem}[Change of train tracks]\label{cambiodibinari}
%\red{Qui non c'è ancora la vicinanza in $H^2$, vedi la nota più sotto nella dimostrazione della Prop 4.2}
Consider two parallel horizontal straight lines in $\R^2$ at distance $h$. 
There exists an embedded curve 
located between 
 the two lines such that it connects the two lines spanning a horizontal interval of length $b\approx 2\sqrt h$ and with energy $\approx 4\sqrt h$, asymptotically as $h\to 0$. More precisely, there exists a regular curve $\gamma:[0,1]\to\R^2$ such that 
\begin{equation*}
\frac{\dot\gamma(0)}{|\dot\gamma(0)|}=\frac{\dot\gamma(1)}{|\dot\gamma(1)|}=e_1,
\quad \gamma(0)=(0,0),\quad \gamma(1) =(b,h)\,,
\end{equation*}
with
\begin{equation*}
b=2\sqrt{h}(1+o(1)),\quad\text{and}\quad \E(\gamma)=4\sqrt{h}(1+o(1))\,.
\end{equation*}
%{\color{red}
% ---------------------------------------------\\
% Moreover, let $\sigma:[0,1]\to \R^2$ be a constant speed parametrization of class $H^2$ of a curve such that $\sigma|_{[0,a]}$ is a straight segment in direction $e_1$, $\sigma(0)=(0,h)$, $\sigma(a')=(b,h)$ for some $a'\in(0,a)$, and consider the joined curve $\alpha:= \sigma * \gamma$. Then, if $\tilde\alpha:[0,1]\to\R^2$ is the constant speed reparametrization of $\alpha$ on $[0,1]$, it holds that
% \begin{equation*}
% \| \tilde\alpha - \sigma \|_{H^2} = o(1),
% \end{equation*}
% as $h\to0$.

% --------------------------------------------------\\
% ((Moreover, if $\sigma_1,\sigma_2:\R\to\R^2$ are $H^2$ regular parametrizations of the two parallel lines with $\sigma_1(0)=(0,0)$ and $\sigma_2(1)=(b,h)$ with $\frac{\dot\sigma_1(0)}{|\dot\sigma_1(0)|}=\frac{\dot\sigma_2(1)}{|\dot\sigma_2(1)|}=e_1$ and $|\dot\sigma_1(0)|=|\dot\sigma_2(1)|$, then there is a reparametrization $\tilde\gamma:[0,\ell(\gamma)/|\dot\sigma_1(0)|]\to\R^2$ of $\gamma$ such that the joined curve $\sigma_2 * \tilde\gamma * \sigma_1$ is a regular curve of class $H^2$ and $\|\tilde\gamma\|_{H^2}^2 \le C(|\dot\sigma_1(0)|)\sqrt{h}(1+o(1))$ as $h\to0$.))
%}
\end{lem}

\begin{proof}
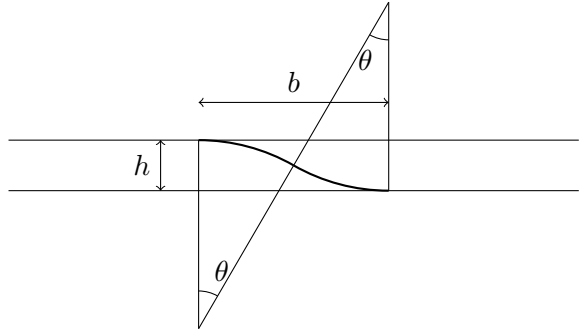
\begin{figure}[h]
\centering
\begin{tikzpicture}[scale=2.5]
\draw[<->] (-0.2,0)--(-0.2,-0.268);
\node[left] at (-0.2,-0.134) {$h$};
\draw[<->] (0,0.2)--(1,0.2);
\node[above] at (0.5,0.2) {$b$};
\draw (-1,0)--(2,0);
\draw (-1,-0.268)--(2,-0.268);

\draw (0,0)--(0,-1);
\draw (0.5,-0.134)--(0,-1);
\draw (1,-0.268)--(1,0.732);
\draw (0.5,-0.134)--(1,0.732);

\draw[thick] (0,0) arc (90:60:1);
\draw[thick] (1,-0.268) arc (-90:-120:1);
\centerarc[](0,-1)(60:90:0.2);
\node[above right] at (0.03,-0.8) {$\theta$};
\centerarc[](1,0.732)(-90:-120:0.2);
\node[below left] at (0.97,0.532) {$\theta$};
\end{tikzpicture}
\caption{The construction of Lemma~\ref{cambiodibinari}}\label{fig:train}
\end{figure}

We construct the curve by
putting together two congruent arcs of circle of radius $1$ as in Figure~\ref{fig:train}. 
The modulus of the curvature is $1$ and the angle is $\theta$. 
The total length is $2\theta$ which also equals the total squared curvature of the arcs. 
The total energy $\mathcal{E}(\gamma)$ is thus $4\theta$. 
The height $h$ and the base $b$ are related to $\theta$ by
\begin{equation*}
h=2(1-\cos\theta)\approx \theta^2 \qquad b=2\sin\theta\approx 2\theta\quad\text{as}\;\;\theta\to 0\,,
\end{equation*}
and therefore given $h$ we can choose $\theta=\arccos\left(1-\tfrac{ h}{2}\right)\approx\sqrt{h}$ to obtain
\begin{equation*}
b\approx 2\sqrt{h}\qquad \E(\gamma)\approx 4\sqrt{h}\,.
\end{equation*}
\end{proof}

The following lemma allows us to suppose, in the construction of the recovery sequence, that
each regular curve of $\N$ is straight near its endpoints,
up to a small perturbation and a small change of the value of its elastic energy.

\begin{lem}[Straightening at vertices]\label{pezzidisegmenti}
Consider a curve $\gamma:[0,1]\to\R^d$ of class $H^2$ parametrized with constant speed. 
Then for every $\eps>0$ there exists a curve $\tilde\gamma:[0,1]\to\R^d$ 
of class $H^2$ parametrized with constant speed such that on some interval $[0,a]$ its image is a straight segment of length at most $\eps$
with direction $\gamma'(0)$, and such that
\begin{equation*}
\gamma(0)=\tilde\gamma(0),
\qquad 
\|\gamma-\tilde\gamma\|_{H^2}\leq \eps, \qquad\mathcal{E}(\tilde\gamma)\leq \mathcal{E}(\gamma)(1+\eps).
\end{equation*}
\end{lem}

\begin{proof}
Without loss of generality we suppose that $\gamma(0)=0$, 
$\dot\gamma(0)=e_1$, and for $\delta>0$ we define
\begin{equation*}
\gamma_\delta(t)=\gamma(t)+\delta \psi(t)e_1 \,,
\end{equation*}
where $\psi:[0,1]\to\R$ is a fixed $C^2$ map with
\begin{equation*}
\psi(t)=1\text{ for }t\in [0,\tfrac14]\qquad \text{and}\qquad \psi(t)=0\text{ for }t\in [\tfrac12,1]\,.
\end{equation*}
Then $\|\gamma-\gamma_\delta\|_{H^2}=\delta\|\psi\|_{H^2}$, and as a consequence $\mathcal{E}(\gamma_\delta)\leq \mathcal{E}(\gamma)(1+C\delta)$ for some constant $C$ that 
depends on $\|\psi\|_{H^2}$.
Let $\alpha_\delta$ be the constant speed parametrization of the concatenation of a straight segment from $0$ to $\delta e_1$ and $\gamma_\delta$. From Lemma \ref{lem:StimaH2} it follows that $\|\alpha_\delta - \gamma_\delta\|_{H^2(0,1)} $ tends to zero as $\delta\to0$. Hence the thesis follows by taking $\tilde\gamma=\alpha_\delta$ for $\delta>0$ sufficiently small.
\end{proof}

We now come to the proof of the recovery sequence.
\begin{proof}[Proof of Proposition~\ref{recovery}]
Consider $\mathcal{N}=(G,\Gamma)\in\mathcal{C}_{\mathrm{Deg}}$
and let $H=H_0=\mathrm{Sing}(\mathcal{N})$.
Without loss of generality we can assume that the regular curves of $\mathcal{N}$ are parametrized with constant speed equal to their length.
For simplicity we can suppose that $H$ is connected and $\Gamma(H)$
is the origin of $\mathbb{R}^d$. In the case of many connected components it is sufficient to repeat the argument for each one of them. By assumption there exist
(real or virtual) tangents $\Tau_G=(\tau^{z,i})_{(z,i)\in \pi^{-1}(V_G)}$ of the graph $G$, such that $H$ is stratified--straight  
with strata
\begin{equation*}
\emptyset=H_q\subset H_{q-1}\subset\ldots\subset H_1\subset H_0=H
\end{equation*}
with $\mathbb{N}\ni q\leq N$, and such that the (real or virtual) tangents of $(H_i,\Sigma_i)$ coincide with the corresponding ones of $\Tau_G$. 

Fix $\eps>0$. Consider the family $\F$ of regular curves of $\N$ that have an endpoint in $H$. For every $\gamma^i\in \F$ we have that $\Gamma(\pi(z,i))=0$ whenever $p=\pi(z,i)\in H$ and moreover by Lemma~\ref{pezzidisegmenti} we can suppose that near $p$ the image of the curve $\gamma^i$ coincides with a straight segment of length $l\leq\eps$, up to adding a small error of order at most
$\eps$ to the energy. 

We consider the first stratum $H=H_0$ and up to rescaling we can suppose that $\Sigma_0(H_0)\subset B_{l^3}(0)$ where we recall that $0=\Gamma(H_0)$. We want to partially desingularize $H$, that is we replace in a neighbourhood of $0$ the completely singular subnetwork $(H_0,\Gamma)$ with the first stratum $(H_0,\Sigma_0)$. We have to make sure to connect each curve to the right junction and with the right angle.

Consider a curve $\gamma^i\in \F$, and suppose for instance that $p=\pi((0,i))\in V_H$ (the case when $\pi((1,i))\in V_H$ is completely analogous). Then there exists $a^{0,i}$ such that $\gamma^i|_{[0,a^{0,i}]}$ is a straight segment of length $l$. In particular $\dot\gamma^i(a^{0,i})$ is parallel to $\dot\gamma^i(0)=\tau^{0,i}$. We now consider the parallel lines with direction $\tau^{0,i}$ passing through the point $\Sigma_0((0,i))$ and through the point $\gamma^i(a^{0,i})$. Their distance is of order $l^3$. We can thus apply Lemma \ref{cambiodibinari} to modify the curve $\gamma^i$ on the interval $[0,a^{0,i}]$ in order to connect $\Sigma_0((0,i))$ to $\gamma^i(a^{0,i})$ with an energy of order $l^{3/2}\leq \eps^{3/2}$ and with outer tangent $\tau^{0,i}$. Finally we reparametrize the modified curve $\gamma^i$ on $[0,1]$ by constant speed equal to its length. By Remark \ref{rem:StimaH2Concatenation} it follows that the modified curve $\gamma^i$ is also close to the original one in $H^2$-norm.

We repeat this process for every curve in $\F$ and every endpoint in $V_H$. By construction the angle condition is still verified for the newly constructed network, but now the stratum $H_0$ is not 
completely singular anymore.
We then repeat the same process for each stratum until we obtain a regular network $\N_\eps$ whose energy is at most $\overline{\mathcal{E}}(\N)+C\eps$ and whose $H^2$ distance from $\mathcal{N}$ is at most $C\eps$.

Since we can do this for any $\eps>0$, choosing $\eps_n=\tfrac{1}{n}$ we obtain a family $\N_n$ of regular networks that approximate $\N$ and thus satisfies the thesis of the theorem.
\end{proof}

\begin{thm}\label{rilassato}
The functional $\overline{\mathcal{E}}$ is the 
 lower semicontinuous envelope
%relaxation 
of the functional 
$\mathcal{E}$ with respect to the weak convergence in $H^2$, 
%in $H^2$, 
that is 
\begin{equation}
\overline{\mathcal{E}}(\mathcal{N}_{\infty})=\inf\left\lbrace \liminf_{n\to\infty}\mathcal{E}(\mathcal{N}_n)\;\vert\; 
\mathcal{N}_n\in\mathcal{C}_{\mathrm{Reg}}\,, \mathcal{N}_n
\overset{H^2}{\rightharpoonup}\mathcal{N}_\infty\right\rbrace \,.
\end{equation}
Moreover the relaxed functional $\overline{\mathcal{E}}$
%Problem~\ref{problemarilassato} 
admits a minimizer.
\end{thm}
\begin{proof}
The existence of minimizers of the functional $\overline{\mathcal{E}}$
follows by a direct method in the Calculus of Variations
combining the lower semicontinuity of $\overline{\mathcal{E}}$ 
and the compactness, that we both gain thanks to Proposition~\ref{regolariconvergonoadeg}
and Proposition~\ref{recovery}.

By definition $\overline{\mathcal{E}}\leq \mathcal{E}$ in the class $\mathcal{C}_{\mathrm{Reg}}$.
Let $\mathcal{F}$ be a lower semincontinuous functional defined in
the class of $N$--networks such that 
$\mathcal{F}\leq \mathcal{E}$ in the class $\mathcal{C}_{\mathrm{Reg}}$.
Proposition~\ref{recovery} implies that $\overline{\mathcal{E}}(\mathcal{N})\geq \mathcal{F}(\mathcal{N})$ 
in  the class of $N$--networks
and this concludes also
the first part of the statement.
\end{proof}

\begin{rem}
The choice of the weak topology in the above relaxation result is the natural one because of the compactness result of Proposition~\ref{regolariconvergonoadeg}.
However we observe that, as is clear from Proposition~\ref{recovery}, a slightly finer result holds: while the liminf inequality $\overline{\mathcal{E}}(\mathcal{N}_{\infty})\leq  \liminf_{n\to\infty}\mathcal{E}(\mathcal{N}_n)$ holds for every sequence $\mathcal{N}_n$ converging \textit{weakly} to $\mathcal{N}_\infty$, the limsup inequality holds with the strong topology, namely, for every $\mathcal{N}_\infty$ there exists a sequence $\mathcal{N}_n$ converging \textit{strongly} to $\mathcal{N}_\infty$ such that $\overline{\mathcal{E}}(\mathcal{N}_\infty)\geq \limsup_{n\to\infty}\mathcal{E}(\mathcal{N}_n)$ (and therefore equality holds for the limit).
Recalling that the lower semicontinuous envelope of a functional $F$ coincides with the $\Gamma$-limit of the constant sequence $F_n=F$, we can here make a comparison with the difference between Mosco convergence and $\Gamma$-convergence.
\end{rem}

%%%%%%%%%%%%%%%%%%%%%%%%%%%%%%%%%%%%%%%%%%%%%%%%%%%
\section{Lower bounds on the Energy}
%%%%%%%%%%%%%%%%%%%%%%%%%%%%%%%%%%%%%%%%%%%%%%%%%%%

Now that we have shown that the relaxed Problem~\ref{problemarilassato}
admits a minimizers,  one may wonder if in some cases the minimizer is regular, that is if
Problem~\ref{problem} has a minimum.  We report here on a special case in which we get the desired result.

\begin{defn}\label{Thetanet}
We call Theta--network any planar network $(G,\Gamma)$ where $G$ is 
a $3$--graph composed of edges $E_1,E_2,E_2$ with the identifications
\begin{align*}
(0,1)\sim (0,2)\sim (0,3)\quad &\text{and}\quad (1,1)\sim (1,2)\sim (1,3)\,,
% \\
% \angle d^{0,i},d^{0,j}=\tfrac{2\pi}{3} \quad &\text{and}\quad \angle d^{1,i},d^{1,j}=\tfrac{2\pi}{3},
\end{align*}
with assigned directions $\{d^{0,i}\}_{i=1}^3$, $\{d^{1,i}\}_{i=1}^3$ such that at any junction the assigned vectors form three (equal) angles equal to $\tfrac23\pi$.
\end{defn}
\begin{figure}[h]
\begin{center}
\begin{tikzpicture}[scale=0.8]
\draw[color=black,scale=1,domain=-2.09: 2.09,
smooth,variable=\t,shift={(0,0)},rotate=0]plot({1.*sin(\t r)},
{1.*cos(\t r)}) ; 
\draw[color=black,scale=1,domain=-2.09: 2.09,
smooth,variable=\t,shift={(0,-1)},rotate=180]plot({1.*sin(\t r)},
{1.*cos(\t r)}) ; 
\draw
(-0.87,-0.5)--(0.87,-0.5);
\path[shift={(0,0)}] 
 (0,-0.5)[above] node{$E^2$}
 (0,1)[below] node{$E^1$}
 (0,-2)[above] node{$E^3$};
\end{tikzpicture}
\end{center}    
\caption{A representation of the graph $G$ of a Theta--network.}\label{Thetanetwork}
\end{figure}
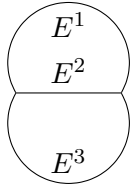
The minimization of the elastic energy among Theta--network
has been considered in~\cite{danovplu}.
It turns out that  in this class the minimizers
of the relaxed problem are regular networks (see~\cite[Theorem~4.10]{danovplu})
and hence Problem~\ref{problem} admits a minimiser. 

Although we have proven that 
in general sequences of regular networks with equibounded energy converge to a
limit network in $\mathcal{C}_{\mathrm{Deg}}$
one may think that for every choice of the topology of the graph $G$ and of the angle condition
minimizers are actually in $\mathcal{C}_{\mathrm{Reg}}$.
The following example shows that even for very simple topologies 
of $G$ this could not be the case and so it is hopeless to give always a positive answer to
Problem~\ref{problem}.

\begin{Example}[The minimizers of the relaxed problem are degenerate]\label{esempiodeg}
Consider the $3$--graph $G$ with assigned angles
composed of $E_1,E_2,E_2$
with $\pi(0,1)=\pi(1,1)=\pi(0,2)$ and $\pi(1,2)=\pi(0,3)=\pi(1,3)$
depicted in Figure~\ref{esempiomindeg}.
Then the length of one of the curves of the minimizers
is  zero.

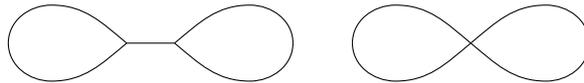
\begin{figure}[h]
\begin{center}
 \begin{tikzpicture}[scale=0.9, rotate=90]
 \draw
 (0,0)--(0,-0.7);
\draw
(0,1.73)to[out= 0,in=90, looseness=1] (0.56,1.1)
(0.56,1.1)to[out= -90,in=50, looseness=1] (0,0);
\draw
(0,1.73)to[out= 180,in=90, looseness=1] (-0.56,1.1)
(-0.56,1.1)to[out= -90,in=130, looseness=1] (0,0);
\draw[rotate=180, shift={(0,0.7)}]
(0,1.73)to[out= 0,in=90, looseness=1] (0.56,1.1)
(0.56,1.1)to[out= -90,in=50, looseness=1] (0,0);
\draw[rotate=180,  shift={(0,0.7)}]
(0,1.73)to[out= 180,in=90, looseness=1] (-0.56,1.1)
(-0.56,1.1)to[out= -90,in=130, looseness=1] (0,0);
\end{tikzpicture}\qquad
 \begin{tikzpicture}[scale=0.9, rotate=90]
\draw
(0,1.73)to[out= 0,in=90, looseness=1] (0.56,1.1)
(0.56,1.1)to[out= -90,in=50, looseness=1] (0,0);
\draw
(0,1.73)to[out= 180,in=90, looseness=1] (-0.56,1.1)
(-0.56,1.1)to[out= -90,in=130, looseness=1] (0,0);
\draw[rotate=180]
(0,1.73)to[out= 0,in=90, looseness=1] (0.56,1.1)
(0.56,1.1)to[out= -90,in=50, looseness=1] (0,0);
\draw[rotate=180]
(0,1.73)to[out= 180,in=90, looseness=1] (-0.56,1.1)
(-0.56,1.1)to[out= -90,in=130, looseness=1] (0,0);
\end{tikzpicture}
\end{center}
\caption{On the left a representation of the graph $G$. On the right a  
possible minimizer $\mathcal{N}_{\min}=(G,\Gamma_{\min})$.}\label{esempiomindeg}
\end{figure}
\end{Example}

We have just seen in Example~\ref{esempiodeg} that in some situations
the minimizers of Problem~\ref{problem} are degenerate networks. 
We want to understand  under which conditions minimizers have 
at least some curves with positive length and do not degenerate to a point.
To be more precise, consider a graph $G$ with assigned angles.
We recall that $\mathcal{D}$ is the set of 
the assigned directions $d^{i,z}$ (see Definition~\ref{defTheta}).
We define the minimization problem
\begin{equation}\label{infalvariarediD}
I_{\mathcal{D}}:=\min\left\lbrace\overline{\mathcal{E}}(\mathcal{N})
\;\vert\: \mathcal{N}=(G,\Gamma)\;\text{is a network}\right\rbrace\,,
\end{equation}
We are interested in conditions on the set $\mathcal{D}$ that
guarantee that $I_{\mathcal{D}}$ is strictly positive.

\begin{lem}\label{2c}
Let $\gamma:[0,\ell(\gamma)]\to\mathbb{R}^2$ be a regular, continuous, piecewise $H^2$ curve
parametrized by arclength
such that $\gamma(0)=0$, 
$\dot{\gamma}(0)=(\cos\theta,\sin\theta)$ for some $\theta\in(-\pi,\pi]$,
and $\dot{\gamma}(\ell(\gamma))=(1,0)$.
Then 
\begin{equation*}
\int_{\gamma}\vert \vec{k}\vert\,\rm{d}s\geq |\theta|\,,
\end{equation*}
and the equality holds if $\gamma$ is convex, i.e., if $\langle \vec{k}, \nu \rangle$ remains non negative (or non positive) along the curve, where $\nu$ is the counterclockwise rotation of $\tfrac\pi2$ of the tangent vector $\tau$ along $\gamma$.\\
Moreover suppose that for every $i\in\{1,\ldots,N\}$
the curve $\gamma^i$ is regular, continuous and of class $H^2$ and suppose that
there exists $\alpha>0$ such that 
\begin{equation*}
\sum_{i=1}^N\int_{\gamma^i}\vert \vec{k}^i \vert\, {\rm{d}}s\geq \alpha\,,
\end{equation*}
then $\mathcal{E}\left(\cup_{i=1}^N\gamma^i\right)\geq 2\alpha$. 
\end{lem}

\begin{proof}
Denote by $\tau$ the unit tangent vector to the curve $\gamma$.
Since the curve $\gamma$ is parametrized by arclength, we have that 
$\dot{\gamma}=\tau$ and $\ddot{\gamma}=\dot{\tau}=\vec{k}$.
By hypothesis the angle spanned by the unit  tangent vector  and the horizontal axis is greater or equal than  $|\theta|$, 
hence
\begin{equation*}
\int_{\gamma}\vert \vec{k}(s)\vert\,\rm{d}s=\int_0^{\ell(\gamma)}\vert \dot{\tau}(s)\vert \,\rm{d}s \geq |\theta|\,.
\end{equation*}
By writing $\tau=(\cos \varphi(s), \sin \varphi (s))$ for some function $\varphi(s) \in (-\pi,\pi]$, as $|\dot\tau|=|\dot\varphi|=|\langle \vec{k}, \nu \rangle|$ we see that equality holds if $\langle \vec{k}, \nu \rangle$ has a sign.

We pass now to prove the second part of the statement.
Using Cauchy--Schwarz inequality we get 
\begin{equation}\label{ineqA}
\left(\sum_{i=1}^N\int_{\gamma^i}\vert \vec{k}^i\vert\,\rm{d}s\right)^2
=\left(\sum_{i=1}^N
\frac{\int_{\gamma^i}\vert \vec{k}^i\vert\,\rm{d}s}{\sqrt{\ell(\gamma^i)}}
\sqrt{\ell(\gamma^i)}\right)^2
\leq 
\left(\sum_{i=1}^N\frac{(\int_{\gamma^i}\vert\vec{k}^i \vert\, {\rm{d}}s)^2}{\ell(\gamma^i)}\right) 
\left(\sum_{i=1}^N\ell(\gamma^i)\right)\,.
\end{equation}
Moreover by H\"{o}lder inequality 
\begin{equation*}
\left(\int_{\gamma^i}\vert \vec{k}^i \vert\, {\rm{d}}s\right)^2\leq 
\ell(\gamma^i)\int_{\gamma^i} \vert \vec{k}^i\vert^2 \, {\rm{d}}s\,,
\end{equation*}
that combined with~\eqref{ineqA} gives
\begin{align*}
\mathcal{E}\left(\cup_{i=1}^N\gamma^i\right)
&=\sum_{i=1}^N\left(\int_{\gamma^i}\vert \vec{k}^i\vert^2\,\rm{d}s
+\ell(\gamma^i)\right)
\geq 
\sum_{i=1}^N\left(
\frac{(\int_{\gamma^i}\vert \vec{k}_i \vert\, {\rm{d}}s)^2}{\ell(\gamma^i)}+\ell(\gamma^i)\right)\\
&\geq  \frac{
\left(\sum_{i=1}^N
\int_{\gamma^i}\vert \vec{k}_i\vert\,\rm{d}s\right)^2}
{\sum_{i=1}^N\ell(\gamma^i)}
+\sum_{i=1}^N\ell(\gamma^i)
\geq\frac{\alpha^2}{\sum_{i=1}^N\ell(\gamma^i)}
+\sum_{i=1}^N\ell(\gamma^i)
\geq 2\alpha\,.
\end{align*}
\end{proof}

\begin{lem}
We have that 
$I_{\mathcal{D}}$ defined in~\eqref{infalvariarediD} equals zero if and only if the graph 
$G$ is stratified-straight.

Moreover, suppose that one of the following conditions holds:
\begin{itemize}
\item[i)]  at every
vertex $p$ the convex hull of the outer tangents contains the origin in
its interior, that is
$0\in Int\left(co\left(\{d_{z,i}\}_{(z,i)\in \pi^{-1}(p)}\right)\right)$,
\item[ii)] two edges $E_i$ and $E_j$ have both vertices in common
and, if $\pi(z_i,i)=\pi(z_j,j)$ for some $z_i,z_j$, then $d^{z_i,i}$ and $d^{z_j,j}$ are linearly independent,
\end{itemize}
then $I_{\mathcal{D}}>0$.
\end{lem}

\begin{proof}
Let us prove the first claim. Suppose first that $I_{\mathcal{D}}=0$. Then there exists a sequence of networks
$\{\mathcal{N}_n\}$ with 
$\widetilde{\mathcal{E}}(\mathcal{N}_n)\to 0$, and in particular $L(\mathcal{N}_n)\to 0$. This means that the networks $\mathcal{N}_n$ are converging strongly, up to translation, to a constant network $\mathcal{N}_\infty$, where every curve is collapsed to a point. In particular, recalling Definition \ref{def:singular}, $\Sing(\mathcal{N}_\infty)=G$. By Proposition~\ref{regolariconvergonoadeg} we obtain that $\mathcal{N}_\infty$ belongs to $\mathcal{C}_{Deg}$. Recalling Definition \ref{degnetwRd}, this implies that $G=\Sing(\mathcal{N}_\infty)$ is stratified straight. This proves one implication.
On the other hand if $G$ is stratified straight, there exists a degenerate immersion
$\Sigma_0$ whose regular curves are straight segments. By Proposition~\ref{recovery}
there exists a sequence $\{\mathcal{N}_n\}$ of regular networks such that
$\Gamma_n \to \Sigma_0$ in $H^2$. 
Analogously there exists recovery sequences for every 
rescaling $\frac{1}{\lambda}\Sigma_0$
of $\Sigma_0$ with $\lambda>>1$. By a diagonal argument one can find
a sequence of regular networks whose energy converges to zero. This proves the opposite implication.
 
Consider any regular network $\N=(G,\Gamma)$ with $G$ graph with assigned angles
encoded by $\mathcal{D}$. 

Suppose now that i) holds.
Since the number of vertices is finite there exists at least one junction $p\in
\Gamma(V_G)$ that lies on the boundary of the convex hull of
$\Gamma(V_G)$, and thus that lies on $\partial H$ where $H$ is a closed
half plane containing $\Gamma(V_G)$. Then at least one of the curves
arriving at $p$, call it $\gamma$, must have an outer tangent that goes
out of $H$ and whose smallest angle with $\partial H$ is at least
$c(\mathcal{D})$ where $c(\mathcal{D})>0$ depends only on $\mathcal{D}$ and it is
strictly positive by assumption. 
Then, by Lemma~\ref{2c}, we have
$\int_{\gamma}\vert \vec{k}\vert\,\rm{d}s\geq c(\mathcal{D})$
and thus
$\mathcal{E}(\gamma)\geq 
\frac{\left(\int_{\gamma}\vert \vec{k}\vert\,\rm{d}s\right)^2}{\ell(\gamma^i)}
+\ell(\gamma^i)\geq \frac{c(\mathcal{D})^2}{\ell(\gamma^i)}
+\ell(\gamma^i)\geq 2c(\mathcal{D})$.

Suppose now that ii) holds. We then have that the least positive angle $\delta \in [0,2\pi)$ between $d^{z_i,i}$ and $d^{z_j,j}$ is actually strictly positive, i.e., $\delta>0$.
Consider the two maps $\gamma^i:=\Gamma_{\vert E_i}$ and $\gamma^j:=\Gamma_{\vert E_j}$, with $\Gamma$ a competitor for $I_{\mathcal{D}}$.
Then by Lemma~\ref{2c} we have that 
$\int_{\gamma^i}\vert \vec{k}^i\vert\,\mathrm{d}s+\int_{\gamma^j}\vert \vec{k}^j\vert\,\mathrm{d}s\geq\delta$,
and thus $\mathcal{E}(\gamma^i)+\mathcal{E}(\gamma^j)\geq 2\delta$.
\end{proof}

%\begin{rem}
%As another application of Lemma~\ref{2c}, it is possible to  
% improve the lower bound
%of the elastic energy of a Theta--network
%that we introduced above in Definition~\ref{Thetanet}, namely if 
%$\Theta$ is a  Theta--network we have that 
%$\mathcal{E}(\Theta)\geq \tfrac{16\pi}{3}$  (cf~\cite[Lemma~2.5.]{danovplu}).
%\end{rem}

As another application of Lemma~\ref{2c} 
we improve the lower bound (cf~\cite[Lemma~2.5.]{danovplu})
of the elastic energy of a Theta--network
that we introduced above in Definition~\ref{Thetanetwork}.

\begin{lem}
Let $\Theta$ be a  Theta--network. Then 
\begin{equation*}
\mathcal{E}(\Theta)\geq \frac{16\pi}{3}\,.
\end{equation*}
\end{lem}
\begin{proof}
In order to prove the required lower bound we can assume that the curves 
$\gamma^1,\gamma^2,\gamma^3$ realize a minimizing Theta-network, 
which exists by~\cite[Theorem 4.10]{danovplu}. 
Then the curves $\gamma^i$ are injective and  real analytic 
by~\cite[Proposition 4.11]{danovplu} and 
by Proposition~\ref{prop:CriticalPoints}.
Without loss of generality we let 
\begin{equation*}
\gamma^1(0)=\gamma^2(0)=\gamma^3(0)=(0,0)\quad\text{ and }\quad 
\gamma^1(1)=\gamma^2(1)=\gamma^3(1)=(a,0)\,,
\end{equation*}
for some $a>0$. 

%Denote by $\alpha^i$ the smallest angle (in modulus)
%between $\frac{\dot{\gamma}^i(0)}{\vert \dot{\gamma}^i(0)\vert}$ 
%and  $(1,0)$.
%Up to relabeling the curves and up to a symmetry with respect to the horizontal axis 
%we can assume that
%\begin{equation*}
%\alpha^1\in [0,\tfrac{\pi}{3}]\quad 
%\alpha^2=\alpha^1+\tfrac23 \pi,\quad \alpha^3=\tfrac23 \pi- %\alpha^1\,.
%\end{equation*}

We denote by $\alpha^i$ the angles between $(1,0)$ and $\tau^i(0)$. Up to relabeling the curves and up to a symmetry with respect to the horizontal axis we can assume that $\alpha^1\in [0,\tfrac{\pi}{3}]$ and that 
\[\alpha^2=\alpha^1+\tfrac23 \pi,\quad \alpha^3=\alpha^1-\tfrac23 \pi.
\]
We claim that for any $i$ there exists $t^i\in [0,1]$
such that $\frac{\dot{\gamma}^i(t^i)}{\vert \dot{\gamma}^i(t^i)\vert}=(1,0)$. We suppose for now that the claim is proven, and proceed to show how to conclude.
We apply Lemma \ref{2c} separately to $(\gamma^i|_{[0,t^i]})_{i=1}^3$ and to $(\gamma^i|_{[t^i,1]})_{i=1}^3$. The total angle spanned by the three curves in the first case is 
\begin{align*}
|\alpha^1|+|\alpha^2|+|\alpha^3|&=|\alpha^1|+\left|\alpha^1+\tfrac23\pi\right|+\left|\alpha^1-\tfrac23\pi\right|\\
&=\alpha^1+\left(\alpha^1+\tfrac23\pi\right)+\left(\tfrac23\pi-\alpha^1\right)\\
&=\alpha^1+\tfrac43\pi\geq \tfrac43\pi
\end{align*}
and therefore by Lemma \ref{2c} the total energy of the first three pieces of curves $(\gamma^i|_{[0,t^i]})_{i=1}^3$ is at least $\tfrac83 \pi$. In a totally analogous way we conclude the same for the final pieces of curves $(\gamma^i|_{[t^i,1]})_{i=1}^3$ and thus we conclude that $\mathcal{E}(\Theta)\geq\tfrac{16}{3}\pi$.

%By Lemma~\ref{2c} we have that
%\begin{align*}
%\sum_{i=1}^3 \mathcal{E}(\gamma^i)\geq \sum_{i=1}^3 2\alpha^i\geq %\tfrac83\pi\,.
%end{align*}
We are left to show how to prove the claim, namely that for any $i$ there exists $t^i\in [0,1]$
such that $\frac{\dot{\gamma}^i(t^i)}{\vert \dot{\gamma}^i(t^i)\vert}=(1,0)$.
Unless the curve $\gamma^i$ is a horizontal straight segment (and in this case the result trivially follows)
by analiticity there are finitely many intersections with the horizontal axis
$(a_j,0)=\gamma^i(t_j)$ with $t_j<t_{j+1}$. Moreover, since $\gamma^i(1)=(a,0)$ with $a>0$,
there exists an index $j$ such that $a_j<a_{j+1}$.
We consider the arc $\sigma(t):=\gamma^i(t)$ with $t\in (t_j,t_{j+1})$ that,
without loss of generality,  lies in the half--plane $\{y>0\}$.
It is possible to complete  $\sigma(t)$ to a smooth simple closed curve 
with $t\in [t_j,s]$ for some $s>t_{j+1}$ such that for every $t\in (t_{j+1},s)$ it holds
$\frac{\dot{\sigma}(t)}{\vert\dot{\sigma}(t)\vert}\neq (1,0)$.
Then by Hopf's Umlaufsatz \cite[Theorem 2.4.7]{AbToCurveSuperfici}  the degree of the tangent map of $\sigma(t)$ is different from zero,
and thus the tangent map is surjective. 
Therefore the value $(1,0)$ must be attained in the interval $[t_j,t_{j+1}]$
and the claim is proved.
\end{proof}

\section{A characterization of degenerate networks in $\mathbb{R}^2$}\label{charact}

Although the definition of the class of degenerate networks 
given in Definition~\ref{degnetwRd} is very neat 
and convenient, it has a disadvantage:
it is based on the validation of the angle condition given by Definition \ref{angolifissatiRd}, and thus on the assumption that there exists a family 
$\Tau$ of real or virtual tangents. 
In general one cannot easily verify this condition.
At least when the ambient space is $\mathbb{R}^2$ we are able to
give an equivalent characterization based on conditions that can be validated 
by a procedure consisting in a finite number of steps. 
The procedure can be roughly summarized as follows: given an $N$-graph $G$ with assigned angles $\mathcal{D}$, to every path $\mathcal{P}$ made of consecutive edges in $G$ we assign an abstract “total turning angle” $\Theta(\mathcal{P})$ that is obtained by summing the turning angles (forced by the angle condition $\mathcal{D}$) at every junction of the path. This angle only depends on the graph structure and on $\mathcal{D}$, but not on a realization of the graph as a network. Then a network $\mathcal{N}=(G,\Sigma)$ satisfies the angle condition given by $\mathcal{D}$ if and only if: for every cycle $C$ (i.e. closed path) entirely contained in $\Sing(\mathcal{N})$ we have that $\Theta(C)$ is a multiple of $2\pi$; for every (open) path $\mathcal{P}$ whose first and last edges are in $\Reg(\mathcal{N})$ but whose other edges are all in $\Sing(\mathcal{N})$ the total angle $\Theta(\mathcal{P})$ coincides with the angle between the real tangents (given by $\Sigma$) at the the second and penultimate vertices of the path.

\medskip

From now on we fix the dimension of the ambient space to $d=2$.

\begin{defn}[Path and cycle]
Let $G$ be an $N$--graph with assigned angles.
A path composed of $\mathcal{J}$ edges is a sequence 
\begin{align*}
\mathcal{P}: \{1,\ldots, \mathcal{J}\} &\to \{0,1\}\times\{1, \ldots, N\}\\
j&\to (z_j,i_j)
\end{align*}
such that for every $j\in\{1,\ldots,\mathcal{J}-1\}$ 
it holds $\pi(1-z_j,i_j)=\pi(z_{j+1},i_{j+1})$.
A path is a \emph{cycle} if $\pi(1-z_{\mathcal{J}},i_{\mathcal{J}})=\pi(z_{1},i_{1})$.
If a path is not a cycle we call it \emph{open path}.
\end{defn}
We should think of $(z_j,i_j)$ as encoding the following fact: when we travel through the path, 
the $j$--th edge is the one with index $i_j$ and its first endpoint that we meet is $z_j$.

When the first element of a path $\mathcal{P}$ is $(z,i)$ and the last is $(w,j)$
sometimes we will simply say that $\mathcal{P}$ is a path from the edge
$E_i$ to the edge $E_j$.

\medskip

Let $\vec{a}$ and $\vec{b}$ two vectors in $\mathbb{R}^2$.
Then the angle between  $\vec{a}$ and $\vec{b}$ , denoted by
$\angle \vec{a},\vec{b}$, is an element of $\mathbb{R}$ mod $2\pi$. From now on and for the rest of the section, an angle $\angle \vec{a},\vec{b}$ will always be understood as an element of $\mathbb{R}$ mod $2\pi$, and a representative for $\angle \vec{a},\vec{b}$ will be always assumed to be the element of $[0,2\pi)$ such that the counterclockwise rotation of $\vec{a}$ of such an angle yields $\vec{b}$.

\begin{defn}[Angle of a path]\label{angolocammino}
Let $G$ be an $N$--graph with assigned angles $\mathcal{D}$ in $\mathbb{R}^2$.
We define the angle of a path $\mathcal{P}$ composed of $\mathcal{J}$ edges as
\[
\Theta(\mathcal{P}):=
\begin{cases}
\sum_{j=1}^{\mathcal{J}-1}\left(\angle -d^{1-z_j,i_j},d^{z_{j+1},i_{j+1}} \right)\; \mathrm{mod} 2\pi
&\text{if}\;\mathcal{P}\;\text{is an open path}\\
\sum_{j=1}^{\mathcal{J}}\left(\angle -d^{1-z_j,i_j},d^{z_{j+1},i_{j+1}} \right)\;\mathrm{mod} 2\pi
&\text{if}\;\mathcal{P}\;\text{is a cycle}
\end{cases}
\]
with the understanding that the indices are modulo $\J$, so that $\J+1=1$.
\end{defn}

\medskip

For the sake of readability when we consider a 
path composed of $\mathcal{J}$ edges 
we often suppose that, up to relabeling and reparametrizing the edges, we have
\begin{align*}
\mathcal{P}: \{1,\ldots, \mathcal{J}\} &\to \{0,1\}\times \{1,\ldots, \mathcal{J}\}\\
j&\to (0,j)
\end{align*}
and that $\pi(1,j)=\pi(0,j+1)$. Moreover if $\mathcal{P}$
is a cycle we require that $\pi(1,\mathcal{J})=\pi(0,1)$.
In this setting
\[
\Theta(\mathcal{P}):=
\begin{cases}
\sum_{j=1}^{\mathcal{J}-1}\left(\angle -d^{1,j},d^{0,j+1} \right)\;\mathrm{mod} 2\pi
&\text{if}\;\mathcal{P}\;\text{is an open path}\\
\sum_{j=1}^{\mathcal{J}}\left(\angle -d^{1,j},d^{0,j+1} \right)\;\mathrm{mod} 2\pi
&\text{if}\;\mathcal{P}\;\text{is a cycle}
\end{cases}
\]
with the understanding that the indices are modulo $\J$, so that $\J+1=1$.
When 
$\mathcal{N}=(G,\Gamma)$ is a network of class $\mathcal{C}_{\mathrm{Reg}}$ 
and we denote by $\gamma^j=\Gamma\vert_{ E_j}$
we have that
$\gamma^j(1)=\gamma^{j+1}(0)$ for $j\in\{1\ldots,\mathcal{J}-1\}$ and, 
 if $\mathcal{P}$ is a cycle,
$\gamma^{\mathcal{J}}(1)=\gamma^1(0)$.

\begin{lem}\label{condizioneangolisingolari}
Let $\{\mathcal{N}_n\}_{n\in\mathbb{N}}$
be a sequence of networks in $\mathcal{C}_{\mathrm{Reg}}$
such that
\begin{equation*}
\limsup_n \mathcal{E}(\mathcal{N}_n)\leq C<+\infty\,.
\end{equation*}
\begin{itemize}
\item Suppose that $\mathcal{P}$ is an open path composed of $\mathcal{J}$ edges
and that $\lim_{n\to\infty}\ell(\gamma_n^j)=0$ for $j\in\{2,\ldots,\mathcal{J}-1\}$.
Then $\Theta(\mathcal{P})=\lim_{n\to\infty}\left(\angle -\tau_n^{1,1},\tau_n^{0,\mathcal{J}}\right)$.
\item Suppose that $\mathcal{P}$ is a cycle composed of $\mathcal{J}$ edges
and that $\lim_{n\to\infty}\ell(\gamma_n^j)=0$ for $j\in\{1,\ldots,\mathcal{J}\}$.
Then $\Theta(\mathcal{P})=0$ (mod $2\pi$).
\end{itemize}
\end{lem}
\begin{proof}
Suppose that $\mathcal{P}$ is an open path.
For every $n\in\mathbb{N}$ we have that 
\begin{align*}
\angle -\tau_n^{1,1},\tau_n^{0,\mathcal{J}}&
=\sum_{j=1}^{\mathcal{J}-1}\angle -\tau_n^{1,j},\tau_n^{0,j+1}
+\sum_{j=2}^{\mathcal{J}-1}\angle \tau_n^{0,j},-\tau_n^{1,j}\\
&
=\sum_{j=1}^{\mathcal{J}-1}\left(\angle -d^{1,j},d^{0,j+1} \right)
+\sum_{j=2}^{\mathcal{J}-1}\angle 
\frac{\dot\gamma^j_n(0)}{\vert\dot\gamma^j_n(0)\vert}, 
\frac{\dot\gamma^j_n(1)}{\vert\dot\gamma^j_n(1)\vert}\,.
\end{align*}
Then thanks to Lemma~\ref{pocaoscillazione}
one gets the desired result passing to the limit $n\to\infty$.

If instead $\mathcal{P}$ is a cycle, then $\tau_n^{0,\J+1}=\tau_n^{0,1}$ and 
so using again by Lemma~\ref{pocaoscillazione} we obtain
\begin{equation*}
\Theta(\P)=\lim_{n\to\infty} \angle -\tau_n^{1,1},\tau_n^{0,1}=0\,,
\end{equation*}
as desired.
\end{proof}

It is possible to give an alternative definition of the angle condition 
for a singular network with respect to Definition~\ref{angolifissatiRd}
in term of paths,  justified by Lemma~\ref{condizioneangolisingolari}.

\begin{defn}[Angle condition for a singular network in term of paths]\label{angolisingolari}
Let $G$ be an $N$--graph with assigned angles.
We say that a network $\N$ satisfies the angle condition if
\begin{itemize}
\item[i)] its regular curves satisfy the angle condition
in the sense of Definition~\ref{angolifissati};
\item[ii)] if $\mathcal{P}$ is a cycle composed of $\mathcal{J}$ edges such that
for every $j\in\{1,\ldots,\mathcal{J}\}$ the edges
$E_j$ are in $\mathrm{Sing}(\mathcal{N})$ 
then $\Theta(\mathcal{P})=0$ (mod $2\pi$);
\item[iii)] if $\mathcal{P}$ is an open path composed of $\mathcal{J}$ edges such that
for every $j\in\{2,\ldots,\mathcal{J}-1\}$ the edges 
$E_j$ are in $\mathrm{Sing}(\mathcal{N})$ 
and $E_1\cup E_{\mathcal{J}}\subset \mathrm{Reg}(\mathcal{N})$
then $\Theta(\mathcal{P})=\angle -\tau^{1,1},\tau^{0,\mathcal{J}}$.
\end{itemize}
\end{defn}

\begin{rem}
Condition $\mathrm{i)}$ can be seen as a particular case of $\mathrm{iii)}$ by taking $\J=2$.
\end{rem}

Clearly one has to modify also Definition~\ref{def:stratstraightRd} 
and Definition~\ref{degnetwRd}
accordingly to the above new definition of angle condition.

\begin{defn}\label{def:stratstraight}
Fix an $N$--graph with assigned angles $G$.
A subgraph $H \subseteq G$ is \emph{stratified-straight} 
if there exists a finite sequence of subgraphs  (called \emph{strata})
\begin{equation*}
\emptyset=H_q\subset H_{q-1}\subset\ldots\subset H_1\subset H_0=H
\end{equation*}
and maps $\Sigma_i:H_i\to \R^2$  such that for $i=0,\ldots,q-1$ 
\begin{itemize}
\item $(H_i,\Sigma_i)$ is a (possibly singular) network 
that satisfies the angle condition
in the sense of Definition~\ref{angolisingolari} and whose curves are (possibly degenerate) straight segments;
\item  $H_{i+1}=\mathrm{Sing}((H_i,\Sigma_i))$.
\end{itemize}
We call \emph{step} of $G$ the minimal $q$ for which the above holds.
\end{defn}

\begin{defn}\label{degnetw}
Let $G$ be an $N$--graph with assigned angles. 
A network $\mathcal{N}=(G,\Gamma)$
is degenerate if
\begin{itemize}
\item it satisfies the angle condition in the sense of Definition~\ref{angolisingolari};
\item the singular part $\mathrm{Sing}(\N)$ is a stratified-straight subgraph 
in the sense of Definition~\ref{def:stratstraight}.
\end{itemize}
\end{defn}

\begin{prop}
Let $G$ be a graph with assigned angles. Suppose that the ambient space is $\mathbb{R}^2$. 
Then a network $\mathcal{N}=(G,\Gamma)$ satisfies 
Definition~\ref{angolisingolari} if and only if it satisfies Definition~\ref{angolifissatiRd} 
for some choice of virtual tangent vectors.
\end{prop}
\begin{proof}
Suppose that $\N$ satisfies the angle condition as in Definition~\ref{angolisingolari}.
The outward tangents of the regular curves of $\N$ satisfy the requests of
Definition~\ref{angolifissatiRd} because of condition i) in Definition~\ref{angolisingolari}.
We have to construct the set of virtual tangents.
%Without loss of generality we can assume that $G$ is connected. Also, let us 
We assume for the moment that at least one curve of $\N$ is regular,
for instance $\Gamma_{\vert E_1}=:\gamma^1$. 
Let $H\subset G$ be a closed connected component of $\mbox{\rm Sing }\mathcal{N}$
and suppose that an endpoint $p_0=\pi(0,1)$ of $E_1$ lies in $H$. 
%We want to construct the virtual tangents at vertices lying in $H$ by using $\gamma^1$ and paths. First 
We define $\tau^{w,j}$ for any $\pi(w,j)=p_0$ with $\gamma^j$ singular by setting $\tau^{w,j}$ equal to the 
counterclockwise rotation of $\tau^{0,1}$ of the angle $\angle d^{0,1}, d^{w,j}$; whenever such $\tau^{w,j}$ 
has been defined, we also set $\tau^{1-w,j}=-\tau^{w,j}$. It follows that~\eqref{eq4} is satisfied at $p_0$. 
Observe that if $\tau^{z,i}, \tau^{y,l}$ are defined by this last step and $\pi(z,i)=\pi(y,l)\neq p_0$, then
\begin{equation}\label{eq5}
\angle \tau^{z,i}, \tau^{y,l} = \angle d^{z,i}, d^{y,l}\,.
\end{equation}
In fact the path $\mathcal{P}$ given by
\begin{equation*}
\mathcal{P}(1)=(1-z,i),\quad\mathcal{P}(2)=(y,l),
\end{equation*}
is a cycle contained in $\mbox{Sing }\mathcal{N}$, and thus by assumption we get that
\begin{equation*}
\begin{split}
	\Theta(\mathcal{P})&=\angle -d^{z,i}, d^{y,l} +\angle -d^{1-y,l},d^{1-z,i}=\angle -d^{z,i}, d^{y,l} +\angle -\tau^{1-y,l},\tau^{1-z,i}=\\&= \angle -d^{z,i}, d^{y,l} +\angle \tau^{y,l},-\tau^{z,i}
	= \pi + \angle d^{z,i}, d^{y,l} + \pi + \angle \tau^{y,l},\tau^{z,i}=\\
	&=2\pi+\angle d^{z,i}, d^{y,l} +2\pi -\angle \tau^{z,i}, \tau^{y,l}=0\;\text{mod}2\pi\,.
\end{split}
\end{equation*}
With this procedure we have defined every virtual tangent at $p_0$ of singular edges with an endpoint at $p_0$, and some virtual tangents at $p\neq p_0$ of singular edges having endpoints at $p$ and $p_0$. 

Let $p\neq p_0$ be now any vertex at which at least one virtual or real tangent $\tau^{z,i}$ is defined. %corresponding to an edge $E_i$ having endpoints at $p$ and $p_0$. By \eqref{eq5}, 
We can perform the very same construction for the still undefined virtual tangents at $p$ using rotations of $\tau^{z,i}$ in place of the original $\tau^{0,1}$. 
By~\eqref{eq5} it follows that~\eqref{eq4} is satisfied at $p$.

Notice that if $\pi(x,a)=\pi(z,i)\neq p_0, \pi(y,b)=\pi(1-z,i)\neq p_0$, and $\pi(1-x,a)=\pi(1-y,b)=p_0$
(and thus we have just constructed $\tau^{z,i}$ as a rotation of $\tau^{x,a}$ and $\tau^{1-z,i}$ as a rotation of $\tau^{y,b}$), then $\tau^{z,i}=-\tau^{1-z,i}$ as desired. In fact considering the cycle
\begin{equation*}
\mathcal{P}(1)=(1-y,b),\quad\mathcal{P}(2)=(1-z,i),\quad\mathcal{P}(3)=(x,a)\,,
\end{equation*}
we get that
\begin{equation}\label{eq:6}
\Theta(\mathcal{P})=\angle -\tau^{y,b}, \tau^{1-z,i} + \angle - \tau^{z,i},\tau^{x,a} + \angle 
-\tau^{1-x,a},\tau^{1-y,b}=0\quad\mbox{mod }2\pi\,.
\end{equation}
On the other hand, since the sum of the exterior angles of a triangle equals $2\pi$, we have that
\begin{equation}\label{eq:7}
\angle -\tau^{y,b}, -\tau^{z,i} + \angle -\tau^{z,i},\tau^{x,a} + \angle -\tau^{1-x,a},\tau^{1-y,b}=2\pi=0\quad\mbox{mod }2\pi\,,
\end{equation}
and subtracting~\eqref{eq:7} to~\eqref{eq:6}
we get that $\angle -\tau^{y,b}, \tau^{1-z,i}= \angle -\tau^{y,b}, -\tau^{z,i} $ mod $2\pi$.\\
Now if $p_1\in H$ is a vertex of an edge having the other endpoint at $p_0$, all the (virtual and real) tangents at $p_1$ are defined. From the above argument, it follows that we can define $\tau^{1-w,j}=-\tau^{w,j}$ for any virtual tangent $\tau^{w,j}$ of a singular edge $E_j$ having an endpoint at $p_1$ without getting contradictions with the other already defined virtual tangents.\\
Therefore it follows that we can then iterate the above arguments, possibly considering cycles passing through $p_0$ that are not triangles, and use the assumption on $\Theta(\mathcal{P})$, in order to check that the iterated construction of virtual tangents does not lead to contradictions. Eventually,  we are able to define the virtual tangents at any vertex of the connected component $H$, getting that Definition~\ref{angolifissatiRd} is satisfied at such vertices.
% QUI commentato quello che c'era prima al posto del rosso qui sorpa
%
% Therefore we see that for any newly constructed $\tau^{w,j}$, 
% we can set $\tau^{1-w,j}=-\tau^{w,j}$ without getting contradictions
% clearly by using the same procedure to define the tangents along open path
% we cannot occur in any contradiction.\\
% Hence eventually, iterating the above argument, we are able to define the virtual tangents at any vertex of the connected component $H$, getting that Definition~\ref{angolifissatiRd} is satisfied at such vertices.
We can apply the same argument to any connected component $H$ of $\mbox{Sing }\mathcal{N}$, completing the implication.\\
If instead it occurs that $\mbox{Sing }\mathcal{N}=G$, we can just choose $\tau^{0,1}=(1,0)=-\tau^{1,1}$ arbitrarily. Then we can perform the very same construction described above.

\smallskip

Conversely, suppose now that Definition~\ref{angolifissatiRd} is satisfied. 
Then obviously the regular curves satisfy the angle condition i) of Definition~\ref{angolisingolari}. 
Moreover, if $\mathcal{P}$ is an open path, then the condition $\tau^{z,i}=-\tau^{1-z,i}$ on the virtual tangents implies point iii) of Definition~\ref{angolisingolari}. 
More generally observe that if the edges seen by a path $\mathcal{P}$ of step $\mathcal{J}$ 
are contained in $\mbox{\rm Sing }\mathcal{N}$, and for simplicity we write $\mathcal{P}(i)=(0,i)$, then
\begin{equation}\label{eq6}
	\sum_{l=l_0}^{L-1} \angle -\tau^{1,l},\tau^{0,l+1}=\angle -\tau^{1,l_0},\tau^{0,L}\quad\mbox{mod }2\pi\,.
\end{equation}
So, finally suppose that $\mathcal{P}$ is a cycle in $\mbox{\rm Sing }\mathcal{N}$ 
and we write $\mathcal{P}(i)=(0,i)$ for simplicity; 
by the fact that every involved tangent is virtual, using~\eqref{eq6} we get that
\begin{equation*}
\begin{split}
	\Theta(\mathcal{P})&=\left(\sum_{i=1}^{\mathcal{J}-1} \angle -\tau^{1,i},\tau^{0,i+1} \right)+ \angle -\tau^{1,\mathcal{J}},\tau^{0,1}= \angle -\tau^{1,1},\tau^{0,\mathcal{J}} + \angle -\tau^{1,\mathcal{J}},\tau^{0,1} =\\&=
	\angle \tau^{0,1},\tau^{0,\mathcal{J}} + \angle \tau^{0,\mathcal{J}},\tau^{0,1} = 0\quad\mbox{mod }2\pi\,,
\end{split}
\end{equation*}
thus completing the equivalence of the definitions.
\end{proof}

We underline the fact that in Definition~\ref{def:stratstraight} we simply replace
the angle condition of Definition~\ref{angolifissatiRd} by the one of Definition~\ref{angolisingolari},
and these replacement directly affect the new Definition~\ref{degnetw}
of degenerate networks.
Then it is clear that taking advantage of the above proposition we are
also able to prove the following:

\begin{cor}
Let $G$ be a graph with assigned angles. Suppose that the ambient space is $\mathbb{R}^2$. 
Then a network $\mathcal{N}=(G,\Gamma)$ is degenerate in the sense of Definition~\ref{degnetwRd}
if and only if it is degenerate in the sense of Definition~\ref{degnetw}.
\end{cor}

\begin{rem}
As we already mentioned, we remark again that it is somehow easier to use Definition~\ref{angolifissatiRd} in the technical arguments. However, in the very remarkable case of dimension $d=2$, 
Definition~\ref{angolisingolari} has the great advantage of being verifiable by an algorithm with finitely many steps. This is clearly not true for the general Definition~\ref{angolifissatiRd}.
\end{rem}

\section{On the relation between straight and stratified straight subgraphs}\label{strstrcoinceideconstr}

In this section we study a simple but remarkable case in which we completely characterize stratified straight or straight subgraphs. This helps us to better understand the algebraic and combinatorial relation between these two concepts, together with providing a non-trivial case in which the two definitions are not equivalent.

In the whole section we study networks in $\mathbb{R}^2$ and we take advantage of the equivalent
characterization of the class of degenerate networks we presented in Section~\ref{charact}:
we use Definition~\ref{angolisingolari}, Definition~\ref{def:stratstraight} and Definition~\ref{degnetw}.

\medskip

Throughout the section we will consider an $N$--graph $G$ with junctions of order at most four and suppose that 
for every junction $p=\pi(z_1,i_1)=\ldots=\pi(z_k,i_k)$ with $k\le 4$
the vectors $d^{z_1,i_1},...,d^{z_k,i_k}$ are distinct and they form angles that are multiples of $\frac\pi2$. In this section, if $\vec{a},\vec{b}$ are two planar vectors, we denote by $\angle \vec{a},\vec{b}\in[0,2\pi)$ the angle described by the counterclockwise rotation of $\vec{a}$ that yields $\vec{b}$.

\begin{rem}
Let $\N=(G,\Gamma)$ be a degenerate network of the type considered above, $H$ a stratified
straight subgraph of $G$ composed of edges $E_1,\ldots,E_{k}$
and $\Sigma:H\to\mathbb{R}^2$ such that
$(H,\Sigma)$ is a  (possibly singular) network 
that satisfies the angle condition
in the sense of Definition~\ref{angolisingolari} and whose curves are (possibly degenerate) straight segments.
We observe that there exist only two possible orthogonal directions, 
identified by two orthogonal unit vectors $a,b$, such that, 
if $\sigma^i:=\Sigma\vert _{E_i}$ is a regular straight segment, 
then $\dot{\sigma}^i$ is parallel to $a$ or $b$. 
In particular, up to rotation, we can assume that $a=(1,0)$ and $b=(0,1)$.
%In fact, suppose that some $\Sigma_0^{{i}}$ is a segment; 
%after a rotation suppose that $(\Sigma_0^{{i}})'=\alpha(1,0)$ with $\alpha>0$. 
%Let $\Sigma_0^j$ be another regular segment. 
%Since $H_0$ is connected, there exists a path $\mathcal{P}$ that 
%we can assume is given by $\mathcal{P}(l)=(0,l)$ such that $\mathcal{P}(1)=(0,i)$ 
%and the final step is $\mathcal{P}(\mathcal{J})=(0,j)$. 
%By hypothesis the map $\Sigma_0$ satisfies the angle condition, hence, 
%since at any junction the angle $\angle -d^{1,l},d^{0,l+1}$ is a multiple of $\frac{\pi}{2}$, 
%%modulo $2\pi$, 
%also the angle $\angle (\Sigma_0^i)',(\Sigma_0^j)'$ is a multiple of $\frac{\pi}{2}$.
\end{rem}

\begin{rem}[Canonical assignment of the vectors $d^{z,i}$]\label{sceltacanonica}
Let $\N=(G,\Gamma)$ be a degenerate network and $H$ a stratified
straight subgraph of $G$ composed of the edges $E_1,\ldots,E_{k}$.
For every $(z,i)\in\{0,1\}\times\{1,\ldots,k\}$ we can give an explicit choice of 
the vector $d^{z,i}$, once a first edge is chosen.
Fix first $d^{0,1}=(1,0)$ for example. For every $(z,i)\in\{0,1\}\times\{1,\ldots,k\}$ we want that
$d^{z,i}\in \{(\pm1,0), (0,\pm1)\}$. We require then that for any $i$ it holds
\begin{equation}\label{eq3}
\begin{split}
d^{0,i}=-d^{1,i}\,.
\end{split}
\end{equation}
Then this choice is well defined for any $i$ and uniquely defines the $d^{z,i}$'s. 
In fact, by considering a path from $\pi(0,1)$ to $\pi(z,i)$ and using the rule~\eqref{eq3}
one can determine uniquely $d^{z,i}$.
Suppose by contradiction that following two different paths $\mathcal{P}$ and $\mathcal{P}'$ 
from $(0,1)$ to $(z,i)$ and using the rule~\eqref{eq3} we get different results 
$d^{(z,i)}=a_{\mathcal{P}}$ and $d^{z,i}=b_{\mathcal{P}'}$. 
Then $\angle a_{\mathcal{P}},b_{\mathcal{P}'} \in\{\tfrac{\pi}{2},\pi,\tfrac{3\pi}{2}\}$.
Call $\mathcal{P}''$ the inverted path of $\mathcal{P}'$ from $\pi(z,i)$ to $\pi(0,1)$.
We can join the two paths  $\mathcal{P}$ and $\mathcal{P}''$ and
the resulting path is a cycle $\mathcal{Q}$ starting from $(0,1)$. 
Then it follows that $\Theta(\mathcal{Q})\in\{\tfrac{\pi}{2},\pi,\tfrac{3\pi}{2}\}$, a contradiction to the fact that $H$ is stratified straight.
\end{rem}

This canonical choice of the $d^{z,i}$ clearly depends on the choice of a starting vertex $\pi(0,1)$ together with the assignment $d^{0,1}=(0,1)$. We can use such canonical choice in order to introduce an
order relation $\preceq$ on the set of vertices.

\begin{defn}\label{relazioendordine}
Let $H$ be a connected stratified straight subgraph of $G$, and suppose that $E_{i_0}\subset H$. Starting from $\pi(0,i_0)$ define $d^{0,i_0}=(1,0)$ and canonically assign the vectors $d^{z,i}$ as specified in Remark \ref{sceltacanonica}. Let $v$ and $w$ be two vertices of $H$.
We say that  $v\preceq w$ if and only if there exists a path 
$\mathcal{P}:\{1,...,\mathcal{J}\}\to\{0,1\}\times \{1,...,N\}\,$ such that
\begin{equation*}
\begin{split}
&\mathcal{P}(1)=(z_1,i_1),\,\mathcal{P}(\mathcal{J})=(z_{\mathcal{J}},i_{\mathcal{J}}),\,\,
\mbox{ with $v=\pi(z_1,i_1)$ and $w=\pi(z_\J,i_\J)$}\,,\\
& d^{\mathcal{P}(i)}\neq(-1,0) \quad\forall i=1,...,\mathcal{J}-1\,,
\end{split}
\end{equation*}
Also, we say that $v\prec w$ if and only if $v\preceq w$ and $w\not\preceq v$.
\end{defn}
Roughly speaking $v\preceq w$ if and only if we can reach $w$ starting from $v$ 
with a path that ``never goes left''. Also observe that the order relation depends on the choice of a certain edge $E_{i_0}$ in the considered subgraph.\\
The order relation defines two subsets of the vertices as defined below.

\begin{defn}\label{def:XeY}
	Let $H$ be a connected stratified straight subgraph of $G$, and suppose that $E_{i_0}\subset H$. Starting from $\pi(0,i_0)$ define $d^{0,i_0}=(1,0)$ and canonically assign the vectors $d^{z,i}$ as specified in Remark \ref{sceltacanonica}. We define
	\begin{equation}
		\begin{split}
			&X(i_0):=\left\{ w\in V_{H} \,\,|\,\,\pi(0,i_0)\prec w \right\},\\
			&Y(i_0):= V_{H}\setminus X(i_0).
		\end{split}
	\end{equation}
\end{defn}

\begin{prop}\label{prop:QuadratiStraight}
Let $G$ be an $N$--graph with assigned angles. 
Suppose that every junction of $G$ has order at most $4$ and that, 
if $p$ is a junction with $\pi^{-1}(p)=\{(z_1,i_1),...,(z_k,i_k)\}$, 
then the vectors $d^{z_1,i_1},...,d^{z_k,i_k}$ form angles equal to $\tfrac{n\pi}{2}$ 
for $n\in\{1,2,3\}$. Let $H\subset G$ be a connected stratified straight subgraph and denote by $(H,\Sigma_0)$ the first stratum of $H$.\\
Suppose that for every $E_{i_0}\subset H\cap \mbox{\rm Sing }(H,\Sigma_0)$ there do not exist cycles $\mathcal{P}=(\mathcal{P}(1),\ldots, \mathcal{P}(\mathcal{J}))$ such that
\begin{align*}
&\mathcal{P}(1)=(z_0,i_0)\\
&d^{(z_0,i_0)}=(1,0)\\
&d^{\mathcal{P}(j)}\neq (-1,0)\quad \forall j=2,\ldots, \mathcal{J}-1.
\end{align*}
Then $H$ is straight.
\end{prop}

\begin{proof}
Denote by $(H,\Sigma_0)$ the first stratum of $H$, i.e. $\Sigma_0:H\to\mathbb{R}^2$ defines a degenerate network, 
the regular curves of $\Sigma_0$ are straight segments and at least
one curve,  say $\sigma_0^{i_0}:={\Sigma_0}\vert_{{E_{i_0}}}$, is degenerate.
We denote by $\sigma_0^i$ the curve $\Sigma_0\vert_{E_i}$.\\
We want to prove that we can modify $(H,\Sigma_0)$ into a new degenerate network $(H,\tilde{\Sigma}_0)$ 
such that $\tilde{\sigma}_0^{i_0}$ is a regular straight segment
and if $\sigma_0^i$ is a regular straight segment then so is $\tilde{\sigma}_0^i$. 
In such a way, since the edges of $H$ are finitely many, 
iterating the argument we conclude that $H$  is straight.\\
In order to simplify the notation, let us write that $i_0=1$.\\
Fix $d^{0,1}=(1,0)$ and consider the order relation $\preceq$ induced by this choice as given by Definition \ref{relazioendordine}. Clearly $Y(1)\neq\emptyset$, in fact $\pi(0,1)\not\prec\pi(0,1)$ by definition. Also we have that $\pi(1,1)\in X(1)\neq\emptyset$; indeed $\pi(0,1)\preceq\pi(1,1)$, $\sigma_0^1$ is degenerate, and if by contradiction $\pi(1,1)\preceq \pi(0,1)$ then there exists a path $\mathcal{P}$ of step $\mathcal{J}$ that for simplicity we denote by
\[
\begin{split}
	&\mathcal{P}(1)=(0,k_0),\\
	&\mathcal{P}(j)=(0,j) \qquad\forall\,j=2,...,\mathcal{J}-1,\\
	&\mathcal{P}(\mathcal{J})=(0,1),
\end{split}
\]
such that $\pi(0,k_0)=\pi(1,1)$ and $d^{0,j}\neq(-1,0)$ for any $j=k_0,2,...,\mathcal{J}-1$.
But since $d^{0,1}=(1,0)$, the cycle given by
	\[
	\begin{split}
	&\mathcal{R}(1)=(0,1),\\
	&\mathcal{R}(j)=\mathcal{P}(j-1) \qquad\forall\,j=2,...,\mathcal{J},
	\end{split}
	\]
	contradicts the hypothesis.\\

Now we construct $\tilde{\Sigma}_0$. Let us first define such map on the set of vertices 
$V_{H}$ by setting
\begin{equation*}
	\tilde{\Sigma}_0(v)=
	\begin{cases}
		\Sigma_0(v) &\text{if}\; v\in Y(1)\,,\\
		\Sigma_0(v) +(\varepsilon,0) & \text{if}\;v\in X(1)\,,
	\end{cases}
\end{equation*}
for some $\varepsilon>0$.
We need to check that we can extend $\tilde{\Sigma}_0$ to the edges in a consistent way.
We claim that for any edge $E_j$,
if $\tilde{\Sigma}_0^j(0)\neq\tilde{\Sigma}_0^j(1)$,
the images  of its boundary points $\tilde{\Sigma}_0^j(0),\tilde{\Sigma}_0^j(1)$
can be connected by a regular straight segment $\tilde{\sigma}_0^j(t)$ such that 
$d^{0,j}=\alpha(\tilde{\sigma}_0^j(1)-\tilde{\sigma}_0^j(0))$ with $\alpha>0$. 
Assuming that the claim is true, then the
map $\tilde{\Sigma}_0$ is defined on every edge in the natural way by connecting 
with a straight segment the image through $\tilde{\Sigma}_0$ of its endpoints.
For $\varepsilon$ small enough, all the regular straight segments of $\Sigma_0$ remain regular.
Moreover,
since $\pi(1,1)\in X(1)$ and $\pi(0,1)\in Y(1)$ then $\tilde{\sigma}_0^1(0)\neq\tilde{\sigma}_0^1(1)$ and thus 
$\tilde{\sigma}_0^1$ is a regular straight segment, and the proof is completed.\\
In order to prove the claim we distinguish two cases, adopting the following notation:
\begin{equation*}
\begin{split}
&A(1)=\{E_i \mbox{ edge of }H\,|\,\, \pi(0,i)\in X(1),\,\pi(1,i)\in X(1)\}\,,\\
&B(1)=\{E_i \mbox{ edge of }H\,|\,\, \pi(0,i)\not\in X(1),\,\pi(1,i)\not\in X(1)\}\,,\\
&C(1)=\{E_i \mbox{ edge of }H\}\setminus(A(1)\cup B(1))\,.
\end{split}
\end{equation*} 
\begin{itemize}
\item[Case 1:] Assume first that $E_j\in A(1)\cup B(1)$. 
Then both endpoints of $E_j$ have been moved or both remained unchanged, that is
\begin{equation*}
	\tilde{\Sigma}_0(0)=\Sigma_0(0) \qquad\mbox{and}\qquad \tilde{\Sigma}_0(1)=\Sigma_0(1)\,,
\end{equation*}
	or
\begin{equation*}
	\tilde{\Sigma}_0(0)=\Sigma_0(0)+(\varepsilon,0) \qquad\mbox{and}\qquad \tilde{\Sigma}_0(1)=\Sigma_0(1)+(\varepsilon,0)\,.
\end{equation*}
If $\Sigma_0(0)=\Sigma_0(1)$, that is $\sigma_0^j$ is degenerate, then $\tilde{\Sigma}_0(0)=\tilde{\Sigma}_0(1)$ as well and $\tilde{\sigma}_0^j$ will be degenerate. 
If otherwise $\Sigma_0(0)\neq\Sigma_0(1)$, that is $\sigma_0^j$ is a regular straight segment, then there is $\alpha>0$ such that $d^{0,j}=\alpha(\Sigma_0(1)-\Sigma_0(0))$, and then $d^{0,j}=\alpha(\tilde{\Sigma}_0(1)-\tilde{\Sigma}_0(0))$ as well and a straight segment $\tilde{\sigma}_0$ satisfies the claim.
\item[Case 2:]We are then left with the case of $E_j\in C(1)$, that is when one of the endpoints of $E_j$ has been moved and the other has not. In such a case, up to relabeling, we can assume  $\pi(1,j)\in X(1)$
(i.e. $\pi(0,1)\prec\pi(1,j)$) and $\pi(0,j)\in Y(1)$. Thus 
\begin{equation*}
\tilde{\Sigma}_0(\pi(0,j))=\Sigma_0(\pi(0,j)) \qquad\mbox{and}\qquad
\tilde{\Sigma}_0(\pi(1,j))=\Sigma_0(\pi(1,j))+(\varepsilon,0)\,.
\end{equation*}
If $d^{0,j}=(1,0)$, the claim is proved. Let us show that this is the case.
We have that $d^{0,j}\neq\pm(0,1)$, otherwise $E_j$ would belong to $A(1)\cup B(1)$,
because $\pi(1,j)\preceq \pi(0,j)\preceq \pi(1,j)$.
Suppose by contradiction that $d^{0,j}=-(1,0)$. 
In this case $d^{1,j}=-d^{0,j}=(1,0)$, and thus $\pi(1,j)\preceq\pi(0,j)$.
By assumption there exists a path $\mathcal{P}$ with $\pi(\mathcal{P}(1))=\pi(0,1),\, \mathcal{P}(\mathcal{J})=(1,j)$ such that $d^{\mathcal{P}(i)}\neq (-1,0)$ for any $i=1,...,\mathcal{J}-1$. 
Extending $\mathcal{P}$ to a longer path by setting $\mathcal{P}(\mathcal{J}+1)=(0,j)$
it follows that $\pi(0,1)\preceq \pi(0,j)$.
Actually we have that $\pi(0,1)\prec\pi(0,j)$, for otherwise, as shown before in the case of $\pi(1,1)$,  
we could construct a cycle $\mathcal{Q}$ of step $\mathcal{K}$ starting at $\mathcal{Q}(1)=(1,j)$ such that $d^{(1,j)}=(1,0)$ and $d^{\mathcal{Q}(k)}\neq(-1,0)$ for any $k=2,...,\mathcal{K}-1$, contradicting the hypothesis.\\
But then we have that $\pi(0,j)\in X(1)\cap Y(1)=\emptyset$, that is impossible. Therefore $d^{0,j}\neq (-1,0)$ and the proof of the claim, and then of the proposition, is completed.	
\end{itemize}
\end{proof}

\noindent The assumption of Proposition \ref{prop:QuadratiStraight} has the advantage of being verifiable just by looking at all the possible cycles starting from the degenerate edges of a stratified straight graph.\\

\noindent In the next example we show that the assumption of Proposition \ref{prop:QuadratiStraight} is also necessary, and thus the statement of such proposition is sharp.

\begin{Example}
	Here we give a simple but remarkable example of a graph $G$ with assigned angles such that every junction of $G$ has order at most $4$ and that, 
	if $p$ is a junction with $\pi^{-1}(p)=\{(z_1,i_1),...,(z_k,i_k)\}$, 
	then the vectors $d^{z_1,i_1},...,d^{z_k,i_k}$ form angles equal to $\tfrac{n\pi}{2}$
	with $n\in\{1,2,3\}$; moreover $G$ is stratified straight, but it is not straight, and in fact the assumption of Proposition \ref{prop:QuadratiStraight} is violated.\\
	Consider the image in $\R^2$ of the regular network drawn in Figure \ref{Fig:Controesempio}.	
		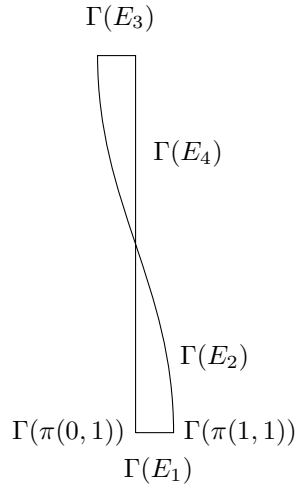
\begin{figure}[H]
			\begin{center}
				\begin{tikzpicture}[scale=1]
				\draw
				(0,0)--(0.5,0)to[out= 90,in=270, looseness=1] (-0.5,5)--(0,5)--(0,0);
				\path[font=\small]
				(0,0)node[left]{$\Gamma(\pi(0,1))$};
				\path[font=\small]
				(0.5,0)node[right]{$\Gamma(\pi(1,1))$};
				\path[font=\small]
				(0.3,-0.5)node{$\Gamma(E_1)$};
				\path[font=\small]
				(0.45,1)node[right]{$\Gamma(E_2)$};
				\path[font=\small]
				(-0.2,5.5)node{$\Gamma(E_3)$};
				\path[font=\small]
				(0.1,3.7)node[right]{$\Gamma(E_4)$};
				\end{tikzpicture}
			\end{center}
	\caption{In the picture we have the image in $\R^2$ of a regular network $(G,\Gamma)$.}\label{Fig:Controesempio}
		\end{figure}
	Such graph $G$ is stratified straight, in fact one can easily construct a sequence of maps $\Gamma_n:G\to\R^2$ such that $(G,\Gamma_n)$ is regular and $\Gamma_n$ converges strongly in $H^2$ to a constant map, i.e. the image of the graph disappears with elastic energy going to zero. However $G$ is not straight and a possible stratification of $G$ is given by
	\[
	H_0=G, \qquad H_1=E_1\cup E_3.
	\]
	
A possible immersion $\Sigma_1$ of the stratum $H_1$
is given by 
two disjoint and horizontal straight segments. In
the couple 
$(H_0,\Sigma_0)$, the immersion
$\Sigma_0$ is given by four curves: $\gamma^2$ and $\gamma^4$ are
two vertical and overlapping straight segments
with the same endpoints,
while $\gamma^1$, $\gamma^3$ are two constant maps, coinciding
with the endpoints of $\gamma^2$ and $\gamma^4$.

	We remark that, in fact, the assumption of Proposition \ref{prop:QuadratiStraight} is not satisfied. More generally we see that as long as a stratified straight subgraph $H$ contains a cycle like the one in Figure \ref{Fig:Controesempio}, then $H$ is not straight, the assumption of Proposition \ref{prop:QuadratiStraight} is not satisfied, and its proof does not work.	
\end{Example}

\section{Side Remarks}\label{sec:SideRemarks}

\subsection{Curves on surfaces/manifolds}

We comment on the fact that the very same kind of definitions about degenerate networks characterize a suitably defined problem for networks of curves into a $2$-dimensional surface in $\mathbb{R}^3$.\\
Fix a $2$-dimensional closed surface $S\subset \R^3$
%compact surface $S\subset \R^3$ without boundary 
and consider a network $\N=(G,\Gamma)$ with $\Gamma:G\to S$. Given a curve $\gamma:I\to S$, the geodesic curvature is given by $\vec{\kappa}_g(t)=\pi_{T_{\gamma(t)}S}\vec{\kappa}(t)$, 
where $\vec{\kappa}(t)$ is the standard curvature of $\gamma$ seen as a curve in $\R^3$ and $T_xS$ is the tangent plane to $S$ at the point $x\in S$. Therefore we define the general elastic energy of the curve $\gamma$ as
\begin{equation}
	\mathcal{E}_{\alpha,\beta}(\gamma) :=\alpha\int_{\gamma} |\vec{\kappa}_g|^2\,ds+\beta L(\gamma).
\end{equation}
Observe that, since the normal component of the curvature vector $\vec{\kappa}$ of $\gamma$ is bounded in terms of the second fundamental form of $S$, then a bound on $\mathcal{E}_{\alpha,\beta}$ actually is a bound on the $L^2$ norm of the whole vector $\vec{\kappa}$.\\
Analogous definitions of angle condition and degenerate network can be given for the class of networks having image in $S$. 
So, the compactness result of Proposition~\ref{regolariconvergonoadeg} can be still easily proved in this case. 
Then, also the proof of the recovery sequence presented in Proposition~\ref{recovery} 
can be adapted to the current situation.
Maintaining the notation of the proof of Proposition~\ref{recovery}, let us say that $H_0\subset\mbox{Sing }\mathcal{N}$ is connected and $\Gamma(H_0)=p\in S$; 
we want to provide an immersion of $H_0$ in $S$.
% in order to construct the recovery sequence. 
Since $S$ is a surface, there exists a local chart $\varphi:U\to\mathbb{R}^2$ at $p$ such that $\varphi$ is isothermal, i.e. the metric tensor $g$ of $S$ can be expressed as $g_{ij}=\lambda^2\delta_{ij}$ on $U$ in the chart $\varphi$. In particular we have that $\varphi$ is a conformal diffeomorphism with its image and its differential preserves angles between tangent vectors. Hence we can construct immersions of $H_0$ in $\varphi(U)$ exactly as in Proposition~\ref{recovery} and then we get the desired recovery sequence by applying $\varphi^{-1}$.

\subsection{Fixed length}\label{fixedlength}

We briefly discuss here an easier variant of Problem~\ref{problem}
of some interest in the applications.

\begin{rem}[Fixed length]
Putting $\alpha=1$ and $\beta=0$ in~\eqref{funzionalegenerale}, then $\widetilde{\mathcal{E}}$ reduces to a functional that we denote by $\mathcal{W}$, by analogy with the Willmore energy.

\begin{prob}\label{lunghezzefissate}
Given an $N$--graph $G$ with assigned angles we consider the minimization problem
\begin{equation*}
\inf\left\lbrace\mathcal{W}(\mathcal{N})\,\vert\, \mathcal{N}=(G,\Gamma)\in\mathcal{C}_{\mathrm{Reg}}\;
\text{with}\; \ell(\mathcal{N}^1),\ldots\ell(\mathcal{N}^N)\;\text{fixed}\right\rbrace\,.
\end{equation*}
\end{prob}

It is easy to prove that Problem~\ref{lunghezzefissate} admits a minimizer
in the class of regular networks by 
a direct method in the Calculus of Variations, 
as we shall now briefly sketch.

Consider a minimizing sequence of networks $\{\mathcal{N}_n\}_{n\in\mathbb{N}}$
composed of curves $\gamma^i_n$.
Combining the bounds~\eqref{normafunz},~\eqref{normader} and~\eqref{normadersec}
together with the fact that the length of each curve is fixed,
we get that  up to subsequence each $\gamma^i_n$
converges to a regular curve $\gamma^i_\infty$ weakly in $H^2$.
The limit networks satisfies the angle condition in the sense of Definition~\ref{angolifissati}
thanks to the strong $C^1$ convergence.
Moreover the functional $\mathcal{W}$ is clearly lower semicontinuous.

\medskip

Notice that fixing the length of each curve avoids any form of degeneracy 
of the limit networks and makes the question on existence of minimizers
trivial.
\end{rem}

\appendix

\section{Critical points}\label{criticalpoints}

We want now to derive the Euler Lagrange equations satisfied by 
\emph{regular} critical points of the functional $\mathcal{E}$. 
%To this aim it is more convenient to work at the level of parametrizations.

Let $\mathcal{N}$ be a regular network 
whose curves are parametrized  by arclength by $\gamma:[0,\ell(\gamma)]\to\mathbb{R}^2$.
We denote by $s$ the arclength parameter of $\gamma^i$.
We recall that $\partial_{s}\gamma=\tau$ and 
$\partial_{s}\tau=\vec{k}$.
% and in $\mathbb{R}^2$ we have that
%$\vec{k}^i=k^i\nu^i$ where $\nu^i$ is the unit normal vector to the curve $\gamma^i$.
For sake of notation 
we introduce the operator $\partial_s^\perp$ that acts on a vector field $\varphi$
giving the normal component of $\partial_s\varphi$ along the curve $\gamma$, that is
\begin{equation*}
\partial_s^\perp\varphi
=\partial_s\varphi-\left\langle \partial_s\varphi,\partial_s\gamma\right\rangle\partial_s\gamma\,.
\end{equation*}
Similarly we call $\partial_s^\parallel\varphi:=\left\langle \partial_s\varphi,\partial_s\gamma\right\rangle\partial_s\gamma
=\left\langle \partial_s\varphi,\tau\right\rangle\tau$.
In $\mathbb{R}^2$ it holds that $\partial_s^\perp\varphi=\left\langle \partial_s\varphi,\nu\right\rangle\nu$, where $\nu$ is the counterclockwise rotation of $\partial_s\gamma$ by an angle equal to $\tfrac{\pi}{2}$.

\medskip

We compute a ``regular" variation  $\mathcal{N}_\varepsilon$
of $\mathcal{N}$. Each curve of  $\mathcal{N}$ is parametrized by 
arclength by $\gamma^i$.
Let us consider $\varepsilon\in\mathbb{R}$ and smooth functions $\psi^{i}:[0,\ell(\gamma^i)]\to\mathbb{R}^2$. This defines the variation $\gamma^i_\varepsilon:[0,\ell(\gamma^i)]\to\mathbb{R}^2$ of a curve $\gamma^i$ by setting $\gamma^i_\varepsilon:=\gamma^i+\varepsilon\psi^i$.
Observe that the variation $\gamma^i_\varepsilon$ is
not necessarily parametrized by arclength. However, for $\varepsilon\in(-\varepsilon_0,\varepsilon_0)$ with $\varepsilon_0>0$ sufficiently small, $\gamma^i_\varepsilon$ is a regular curve with tangent vector $\tau^i_\varepsilon=\frac{\partial_s\gamma^i
+\varepsilon\partial_s\psi^i}{\vert \partial_s\gamma^i+\varepsilon\partial_s\psi^i \vert}$.

Consider a junction 
of order $m$ of $\mathcal{N}$, so that
\begin{equation*}
\gamma^{i_1}(z_1)=\ldots=\gamma^{i_m}(z_m)\,,
\end{equation*}
with 
$(z_1,i_1),\ldots,(z_m,i_m)\in \{0,\ell(\gamma^{i_1}), \ldots,\ell(\gamma^{i_m})\}\times\{1,\ldots,N\}$ 
all distinct. 
%We are going to take a variation involving the curves $\gamma^{i_1},...,\gamma^{i_m}$.
%\\
The scalar product of unit tangent vectors is given by
\begin{equation*}
\left\langle \tau^{i_1}(z_1),\tau^{i_2}(z_2)\right\rangle=c^{1,2}\,, 
\ldots\,,\, 
\left\langle  \tau^{i_{m-1}}(z_{m-1}),\tau^{i_m}(z_m)\right\rangle=c^{m-1,m}\,.
\end{equation*}

\smallskip

To generate admissible competitors for the Problem~\ref{problem}
for any $\varepsilon$ small enough, 
we need to require that
\begin{equation*}
\gamma_\varepsilon^{i_1}(z_1)=\ldots=\gamma_\varepsilon^{i_m}(z_m)\,,
\end{equation*}
together with the fact that the angle condition is preserved,
% for any $\varepsilon$ small, 
that is to say that for every $i,j,z,w$ such that $\gamma^i(z)=\gamma^j(w)$ it holds that
\begin{equation*}
\frac{\mathrm{d}}{\mathrm{d}\varepsilon}\left\langle\tau_\varepsilon^{i}(z),\tau_\varepsilon^{j}(w)\right\rangle=0\,,
\end{equation*}
for every $\varepsilon\in(-\varepsilon_0,\varepsilon_0)$.\\
First this implies that the variation fields $\psi^i$ satisfy that
\begin{equation}\label{proprieta1}
\psi^{i_1}(z_1)=\ldots=\psi^{i_m}(z_m)\,.
\end{equation}
Secondly, 
%we have the following requirement on the fields $\psi^i$. 
writing
\begin{equation*}
\partial_{s}\psi^{i}=\partial_{s}^\perp\psi^{i}+\partial_{s}^\parallel\psi^{i}
=\left\langle \partial_{s}\psi^i,\nu^i\right\rangle\nu^i
+\left\langle \partial_{s}\psi^i,\tau^i\right\rangle\tau^i
=:\overline{\psi}^i_{s}\nu^i+\widetilde{\psi}^i_s\tau^i\,,
\end{equation*}
a direct calculation yields
\begin{align*}
0&=\frac{\mathrm{d}}{\mathrm{d}\varepsilon}\left\langle\tau_\varepsilon^{i}(z),\tau_\varepsilon^{j}(w)\right\rangle =
\left\langle \frac{\partial_{s}\psi^i-\langle \partial_{s}\psi^i,\tau^i_\varepsilon\rangle \tau^i_\varepsilon}{|\partial_{s}\gamma^i+\varepsilon\partial_{s}\psi^i|}\bigg|_z, \tau_\varepsilon^j(w)  \right\rangle +
\left\langle  \tau_\varepsilon^i(z) , \frac{\partial_{s}\psi^i-\langle \partial_{s}\psi^i,\tau^i_\varepsilon\rangle \tau^i_\varepsilon}{|\partial_{s}\gamma^i+\varepsilon\partial_{s}\psi^i|}\bigg|_w \right\rangle\,.
\end{align*}
Evaluating at $\varepsilon=0$ we get
\begin{align*}
0&=
\left\langle
\partial_{s}^\perp\psi^i(z), \tau^j(w)\right\rangle
+\left\langle \tau^i(z),\partial_{s}^\perp\psi^j(w)\right\rangle\\
&=\overline{\psi}_{s}^{i}(z)\left\langle\nu^i(z),\tau^j(w) \right\rangle
+\overline{\psi}_{s}^{j}(w)\left\langle \tau^i(z), \nu^j(w)\right\rangle\\
&=\overline{\psi}_{s}^{i}(z)\left\langle\nu^i(z),\tau^j(w) \right\rangle
-\overline{\psi}_{s}^{j}(w)\left\langle\nu^i(z),\tau^j(w) \right\rangle\,.
\end{align*}
Therefore as a second requirement on the fields $\psi^i$ we impose that
\begin{equation}\label{proprieta2}
\overline{\psi}_{s}^{i_1}(z_1)=\ldots=\overline{\psi}_{s}^{i_m}(z_m)\,.
\end{equation}

\begin{defn}
Consider a regular network $\mathcal{N}=(G,\Gamma)$
composed of the curves $\gamma^i$  parametrized by 
arclength
with $i\in\{1,\ldots,N\}$,
$\varepsilon\in\mathbb{R}$ and smooth functions $\psi^{i}:[0,\ell(\gamma^i)]\to\mathbb{R}^2$.
We say that $\mathcal{N}_\varepsilon$
composed of the curves $\gamma^i_\varepsilon=\gamma^i+\varepsilon\psi^i$
 is a regular variation
of $\mathcal{N}$ if the functions $\psi^i$
satisfy the properties~\eqref{proprieta1} and~\eqref{proprieta2}.
\end{defn}

\begin{defn}
A regular network $\mathcal{N}=(G,\Gamma)$ is a 
critical point for the functional $\mathcal{E}$
if for every regular variation  $\mathcal{N}_\varepsilon$
of $\mathcal{N}$ it holds
\begin{equation*}
\frac{d}{d\varepsilon}\mathcal{E}(\mathcal{N}_\varepsilon)_{|\varepsilon=0}=0\,.
\end{equation*}
\end{defn}

We recall that the oriented curvature of a planar curve $\gamma$ is defined as the 
scalar $k$ such that $\vec{k}=k\nu$ where the unit normal vector
$\nu$ is the counterclockwise rotation
of $\frac{\pi}{2}$ of the unit tangent vector $\tau$ to the curve $\gamma$.

%\begin{defn}
%Let $\varphi$ be a smooth vector field along $\gamma$.
%Then 
%$\partial_s^\perp\varphi:= \partial_s\varphi-\left\langle\partial_s^\perp\varphi,\tau\right\rangle\tau$.
%\end{defn}

\begin{prop}\label{prop:CriticalPoints}
Let $\mathcal{N}=(G,\Gamma)$ be a critical point for $\mathcal{E}$. 
Then the arclength parametrization $\gamma^i:[0,\ell(\gamma^i)]\to\mathbb{R}^2$ of any curve of $\mathcal{N}$ is real analytic and satisfies the equation
\begin{equation}\label{eq2}
 2 (\partial^\perp_{s})^2\vec{k}^{i}
 + \vert\vec{k}^i\vert^2 \vec{k}^{i}
 -\vec{k}^{i}=0 \qquad\mbox{ on }(0,\ell(\gamma^i))\,,
\end{equation}
or, equivalently, in terms of the oriented curvature
\begin{equation*}
2 \partial^2_{s} k^i+\left(k^i\right)^{3} - k^i=0 \qquad \mbox{ on }(0,\ell(\gamma^i))\,.
\end{equation*}
Also, the curves satisfy the following boundary conditions. 
If $p=\pi(z_1,i_1)=...=\pi(z_m,i_{m})$ is a junction of order $m$ of $\mathcal{N}$, then
\begin{align}
\sum_{\substack{(z_j,i_j)\in\pi^{-1}(p)\,: \\ z_j=\ell(\gamma^{i_j})}} k^{i_j}(\ell(\gamma^{i_j}))&=\sum_{\substack{(z_j,i_j)\in\pi^{-1}(p)\,: \\ z_j=0}} k^{i_j}(0)\,,\label{boundarycond1}\\
\sum_{\substack{(z_j,i_j)\in\pi^{-1}(p)\,: \\ z_j=\ell(\gamma^{i_j})}} \big[2 \partial_{s}^\perp \vec{k}^{i_j} +(k^{i_j})^2\tau^{i_j} -\tau^{i_j}\big]\big|_{\ell(\gamma^{i_j})}&=\sum_{\substack{(z_j,i_j)\in\pi^{-1}(p)\,: \\ z_j=0}} \big[2 \partial_{s}^\perp \vec{k}^{i_j} +(k^{i_j})^2\tau^{i_j} -\tau^{i_j}\big]\big|_{0}\,.\label{boundarycond2}
\end{align}
\end{prop}

\begin{rem}
Let us observe that the boundary conditions~\eqref{boundarycond1} and~\eqref{boundarycond2} 
%of a critical point $\mathcal{N}$ for the energy $\mathcal{E}$ given by Proposition \ref{prop:CriticalPoints} 
do not depend on the parametrizations of the curves $\gamma^i$. More precisely, even if such boundary conditions are expressed in terms of quantities evaluated at $0$ or at $\ell(\gamma^i)$, if a curve $\gamma^j$ is reparametrized into the new arclength parametrized immersion $\tilde{\gamma}^j(t)=\gamma^j(\ell(\gamma^j)-t)$, then
\begin{align*}
\tilde{k}^j(\ell(\gamma^j))=-k^j(0)\,, \; &\,\qquad \tilde{k}^j(0)=-k^j(\ell(\gamma^j))\,,\\
\big[2 \partial_{\tilde{s}^j}^\perp \vec{\tilde{k}}^j +(\tilde{k}^j)^2\tilde{\tau}^j -\tilde{\tau}^j\big]\big|_{\ell(\gamma^j)} 
&=-\big[2 \partial_{s^j}^\perp \vec{k}^j +(k^j)^2\tau^j -\tau^j\big]\big|_{0}\,, \\
\big[2 \partial_{\tilde{s}^j}^\perp \vec{\tilde{k}}^j +(\tilde{k}^j)^2\tilde{\tau}^j -\tilde{\tau}^j\big]\big|_{0} 
&= - \big[2 \partial_{s^j}^\perp \vec{k}^j +(k^j)^2\tau^j -\tau^j\big]\big|_{\ell(\gamma^j)}\,,
\end{align*}
where symbols with the tilde identify the obvious geometric quantities 
in terms of the parametri--zation
%related to the curve
 $\tilde{\gamma}^k$. 
%Hence we see that the equalities expressed by the boundary conditions of Proposition \ref{prop:CriticalPoints} are actually independent of parametrizations.
\end{rem}

\begin{proof}[Proof of Proposition~\ref{prop:CriticalPoints}]
In order to calculate the first variation of the functional, 
we can fix a junction $p$ of $\mathcal{N}$ of order $m$ and consider a regular variation given by fields $\psi^{i_j}$ for $j=1,...,m$ such that: if $\pi(z_j,i_j)=p$ but $\pi(1-z_j,i_j)\neq p$, 
then $\psi^{i_j}\equiv0$ in a neighborhood of $1-z_j$. 
By direct computations (for the details for example see~\cite{daplu}) one shows
 that the curves are of class $C^\infty$ and they satisfy the first variation formula
\begin{align*}\label{boundarypart} 
 \frac{d}{d\varepsilon}\mathcal{E}(\mathcal{N}_\varepsilon)\vert_{\varepsilon=0}
 &=\sum_{j=1}^{m} 
 \int_{\gamma^{i_j}}  \left\langle
 2 (\partial^\perp_{s})^2 \vec{k}^{i_j}+ |\vec{k}^{i_j}|^2 \vec{k}^{i_j} 
 -\vec{k}^{i_j},
\psi^{i_j} \right\rangle \,\mathrm{d}s\\
 &\quad+\sum_{j=1}^{m} \left[ 2 \left. \langle \vec{k}^{i_j},
 \partial_{s} \psi^{i_j}\rangle \right|_0^{\ell(\gamma^{i_j})} 
+  \left. \langle -2\partial^\perp_{s} \vec{k}^{i_j}
- |\vec{k}^{i_j}|^2\tau^{i_j}+\tau^{i_j}, \psi^{i_j}\rangle \right|_0^{\ell(\gamma^{i_j})} 
 \right] =0 \,.
\end{align*}
This immediately leads to the interior equations~\eqref{eq2}.\\
In order to get the boundary conditions, we can first set
\begin{equation*}
\psi^{i_j}(0)=\psi^{i_j}(\ell(\gamma^{i_j}))=0\,,
\end{equation*}
so that
\begin{equation*}
\sum_{j=1}^m \overline{\psi}_{s}^{i_j} k^{i_j}\big|_0^{\ell(\gamma^{i_j})}=0\,,
\end{equation*}
and by~\eqref{proprieta2} we get the first boundary condition~\eqref{boundarycond1}
Similarly, letting $\overline{\psi}_{s}^{i_j}(0)=\overline{\psi}_{s}^{i_j}(\ell(\gamma^{i_j}))=0$ 
and using~\eqref{proprieta1}, the second boundary condition~\eqref{boundarycond2} is achieved.\\
Now we prove that each $\gamma^i$ is an analytic curve. 
Fix $i=1$, and suppose that $\pi(0,1)$ is a junction of order $m$. 
The first boundary condition equation gives that
\begin{equation*}
k^1(0)= \sum_{\substack{(z_j,i_j)\in\pi^{-1}(p)\,: \\ z_j=\ell(\gamma^{i_j})}} k^{i_j}(\ell(\gamma^{i_j}))-\sum_{\substack{(z_j,i_j)\in\pi^{-1}(p)\,: \\ z_j=0,\,j\neq 1}} k^{i_j}(0)=:C_1\,.
\end{equation*}
While, since $\langle \partial_{s}^\perp \vec{k}^1(0), \nu^1(0)\rangle=\partial_{s} k^1(0)$, 
by multiplying the second boundary condition equation 
%of $\gamma^1$ 
by $\nu^1(0)$ we can write that
\begin{equation*}
\begin{split}
\partial_{s} k^1(0) &= \frac{1}{2} \bigg\langle \sum_{\substack{(z_j,i_j)\in\pi^{-1}(p)\,: \\ z_j=\ell(\gamma^{i_j})}} \big[2 \partial_{s}^\perp \vec{k}^{i_j} +(k^{i_j})^2\tau^{i_j} -\tau^{i_j}\big]\big|_{\ell(\gamma^{i_j})}+\\
	&\qquad-\sum_{\substack{(z_j,i_j)\in\pi^{-1}(p)\,: \\ z_j=0,\,j\neq 1}} \big[2 \partial_{s}^\perp \vec{k}^{i_j} +(k^{i_j})^2\tau^{i_j} -\tau^{i_j}\big]\big|_{0} , \nu^1(0) \bigg\rangle=:C_2\,.
\end{split}
\end{equation*}
Therefore a direct application of the Cauchy-Kovaleskaya Theorem 
(see for instance~\cite[p.240]{evans}) on the Cauchy problem
\begin{equation*}
\begin{cases}
2\partial_{s}^2k^1+(k^1)^3-k^1=0 & \mbox{ on }[0,\varepsilon)\,,\\
k^1(0)=C_1\,,\\
\partial_{s}k^1(0)=C_2\,,
\end{cases}
\end{equation*}
gives that $k^1$ is real analytic on some interval $[0,\varepsilon)$. 
Analogously, the same holds on $(\ell(\gamma^1)-\varepsilon,\ell(\gamma^1)]$, 
and therefore $k^1$ is real analytic. 
Writing $\tau^1(s)=\left(\cos(\theta(s)),\sin(\theta(s)\right)$ we get that
\begin{equation*}
\partial_{s}\tau^1=(-\sin(\theta),\cos(\theta))\partial_{s}\theta\,,
\end{equation*}
so that
\begin{equation*}
	k^1(s)=\partial_{s}\theta(s).
\end{equation*}
Hence $\theta$ is analytic, that implies that $\partial_{s}\gamma^1=\tau^1$ is analytic, and so is $\gamma^1$.

By the arbitrary of the choice of $\gamma^1$, the result follows for each $\gamma^i$.
\end{proof}

\bibliographystyle{amsplain}
\bibliography{degenerate-elastic-networks}

\end{document}